\documentclass[final,onefignum,onetabnum]{siamart250211}

\ifpdf
\hypersetup{
	pdfauthor={Cyprien Tamekue, ShiNung Ching},
	pdfsubject={Paper manuscrit},
	pdftitle={Control of neural field equations with step-function inputs},
	pdfkeywords={Nonlinear control, control synthesis, step controllability, neural field equations, spatially forced pattern forming system, predicting visual illusions}
}
\fi

\graphicspath{{./figs/}}

\usepackage{bbm}
\usepackage{graphicx}
\usepackage{xcolor}
\usepackage{mathrsfs}
\usepackage{marvosym}
\usepackage{comment}
\usepackage{amsfonts, amssymb}
\usepackage{mathtools}

\usepackage{multirow}
\usepackage{wrapfig}
\usepackage[font=small,labelfont=bf]{caption}
\usepackage{subcaption}

\usepackage{tikz,tkz-tab} 
\usepackage{moreverb}

\usepackage{csquotes}

\numberwithin{equation}{section}
\newcommand{\bs}{\symbol{92}}
\newcommand*\samethanks[1][\value{footnote}]{\footnotemark[#1]}

\newsiamremark{remark}{Remark}
\newsiamremark{hypothesis}{Hypothesis}
\newsiamremark{notation}{Notation}
\newsiamremark{assumption}{Assumptions}
\crefname{hypothesis}{Hypothesis}{Hypotheses}

\def\cF{{\mathcal F}}

\def\N{{\mathbb N}}
\def\Z{{\mathbb Z}}
\def\R{{\mathbb R}}

\def\C{{\mathbb C}}

\def\idty{{\operatorname{Id}}}
\def\dist{{\operatorname{dist}}}

\def\cA{{\mathcal A}}    \def\cS{{\mathcal S}}         \def\cO{{\mathcal O}}  \def\cD{{\mathcal D}}          \def\cF{{\mathcal F}}    \def\cX{{\mathcal X}} \def\cY{{\mathcal Y}}

\usepackage{hyperref}

\headers{Control of neural field equations with step-function inputs}{C. Tamekue and S. Ching}
\title{Control of neural field equations with step-function inputs\thanks{\funding{This work is partially supported by grant R21MH132240 from the US National Institutes of Health to SC.}}}
\author{Cyprien Tamekue\thanks{Department of Electrical and Systems Engineering, Washington University in St. Louis, St. Louis, 63130, MO, USA. (\email{cyprien@wustl.edu}, \email{shinung@wustl.edu}).}  \and ShiNung Ching\samethanks}

\begin{document}

\maketitle

\begin{abstract}
Wilson-Cowan and Amari-type models describe the nonlinear dynamics of interacting neural populations, providing a fundamental framework for modeling how sensory and other exogenous inputs shape neural activity. This article investigates the controllability properties of Amari-type neural field equations subject to piecewise constant-in-time and constant-in-time inputs. The model describes the time evolution of polarization in neural tissue within a spatial continuum, with synaptic interactions modeled by a convolution kernel. We study the synthesis of constant-in-time and piecewise constant-in-time inputs to achieve two-point boundary-type control objectives, namely, steering neural activity from an initial state to a prescribed target state. This approach is particularly relevant for predicting the emergence of paradoxical neural representations, such as discordant visual illusions elicited by overt sensory stimuli. We first present a control synthesis based on the Banach fixed-point theorem, yielding an iterative construction of a constant-in-time input under minimal regularity assumptions on the kernel and transfer function; however, it has practical limitations, even in the linear case. To overcome these challenges, we then develop a generic synthesis framework based on the flow of neural dynamics drift, enabling explicit piecewise constant and constant-in-time inputs. Extensive numerical results in one and two spatial dimensions confirm the effectiveness of the proposed syntheses and show that they consistently outperform inputs derived from naive linearization at the initial or target states when these states are not equilibria of the drift dynamics. By providing a mathematically rigorous framework for controlling delay-free Amari-type neural fields, this work advances our understanding of nonlinear neural population control with potential applications in computational neuroscience, psychophysics, and neurostimulation.
\end{abstract}
\begin{keywords}
Nonlinear control, control synthesis, step controllability, neural field equations, spatially forced pattern forming system, predicting visual illusions.
\end{keywords}
\begin{MSCcodes}
93C20, 92C20, 93C10, 35Q92, 93B50.
\end{MSCcodes}
\section{Introduction}

Wilson-Cowan \cite{wilson1973mathematical} and Amari-type \cite{amari1977dynamics} equations are neural field models that describe the nonlinear dynamics of interacting populations of excitatory and inhibitory neurons at the scale of a cortical macrocolumn (a roughly 1mm diameter section of the cortex). These models are favored for their balance of biophysical interpretability and mathematical tractability. They have been instrumental in exploring a range of questions in theoretical neuroscience, such as how sensory inputs drive neural activity, particularly in visual perception contexts \cite{bertalmio2020visual,bolelli2025neural,bressloff2001geometric,ermentrout1979mathematical,nicks2021understanding,rule2011model,tamekue2023cortical,tamekue2024mathematical,tamekue2025reproducibility}. For example, the steady-state solutions of Amari-type equations have been used to mechanistically model visual illusions, such as the visual MacKay effect or Billock and Tsou experiments \cite{bolelli2025neural,tamekue2023cortical,tamekue2024mathematical,tamekue2025reproducibility}, as asymptotic outcomes of cortical activity driven by sensory inputs.

An important emerging question in the analysis of neural field models concerns their controllability properties \cite{tamekue2024mathematical}, i.e., the extent to which they are labile to exogenous inputs. Such analysis provides a mathematical foundation for both scientific and engineering focus. For instance, determining how sensory inputs drive the formation of neural representations or how external stimulation of the brain may modify brain dynamics in clinical contexts. This perspective advances the theoretical understanding of neural dynamics and opens new avenues for practical applications, such as predicting or designing visual stimuli that induce targeted perceptual effects.

In this work, we study control analysis and synthesis for neural field equations of the Amari type with constant sensory inputs over time. The considered equation reads as
 \begin{equation}\label{eq:NF-intro}\tag{NF}
 \begin{split}
 	\partial_t a(x,t) + \alpha a(x,t) - \mu\int_{\R^d}\omega(x,y)f(a(y,t))dy &= I(x),\qquad (t,x)\in\R_{+}\times\mathbb R^d\\
a(x,0) &= a_0(x),\qquad x\in\R^d.
  \end{split}
 \end{equation}
Here, $d\in\N^{*}$ (but it is usually taken to be one or two) where $\N^{*}$ stands for the set of positive integers, and $\R_{+}:=[0,\infty)$. 
 The time-evolution of the average membrane potential at time $t\ge 0$ is given by the map $x\mapsto a(x,t)$, and the state of cortical activity at time $t=0$ is assumed to be provided by the function $a_0$. Here, $\alpha>0$ is the time scale of activity decay, $I$ is a constant in time input, and the transfer function $f$ captures the nonlinear response of neurons after activation. The parameter $\mu>0$ measures the strength of neuronal interaction while the kernel $\omega(x,y)=\omega(x-y)$ models the synaptic connections between neurons at two different locations $x$ and $y$. We refer to \cite{cook2022neural,coombes2023next} for a recent overview of neural population models. 
 
Although~\eqref{eq:NF-intro} is one of the most parsimonious neural population models widely used to mechanistically describe diverse neurological and visual phenomena (see, e.g.,~\cite{bressloff2011spatiotemporal, ermentrout1998neural} for reviews), it inherently neglects the finite time required for stimulus---represented by the input~$I$---to propagate between neurons at different cortical locations. Specifically,~\eqref{eq:NF-intro} assumes infinite axonal transmission speed between a presynaptic neuron at~$y$ and a postsynaptic neuron at~$x$. In reality, conduction velocities vary across neural populations and typically range from about~$1$~m/s to~$100$~m/s~\cite{shepherd2004synaptic}. Incorporating such finite transmission delays gives rise to delayed neural field equations, which have been extensively studied both theoretically and numerically; see, for instance,~\cite{chaillet2017robust, coombes2009delays, faye2010new, sequeira2022numerical,veltz2013interplay}. Notably, transmission delays play a crucial role in the emergence of oscillatory dynamics associated with neurological disorders, such as Parkinson's disease~\cite{asadi2024dynamics, pasillas2013delay}.

In this work, we deliberately focus on the delay-free formulation~\eqref{eq:NF-intro} for two main reasons. First, despite its simplicity, it still captures a wide range of neural processes. Second, establishing a control framework in this setting provides a robust foundation for addressing more complex models that include state-dependent delays.

It is also worth noting that a more realistic model would treat cortical tissue as a bounded domain~$\Omega \subset \mathbb{R}^d$ rather than the entire space~$\mathbb{R}^d$; see, for example,~\cite[Section~7.1]{veltz2010local}. To mitigate this idealization, we assume that the connectivity kernel~$\omega$ decays rapidly at infinity, resembling a ``Mexican hat'' or ``wizard hat'' shape.

Finally, because our control synthesis is formulated in terms of the flow maps generated by the neural field dynamics, the results naturally extend to inhomogeneous kernels (i.e., $\omega(x,y) \neq \omega(x - y)$) or bounded domains~$\Omega$, provided suitable assumptions on~$\omega$ ensure the applicability of the synthesis framework.

To the authors' knowledge, the controllability and stability properties of neural field equations like~\eqref{eq:NF-intro} have been considered for absolute stability of homogeneous solutions~\cite{faugeras2008absolute}, for feedback stabilization of deep-brain stimulation \cite{brivadis2023existence,chaillet2017robust} (see also \cite{annabi2025activity,brivadis2024adaptive} for observability properties of these equations) or for open-loop control of traveling waves \cite{ziepke2019control}. We also refer to~\cite[p.~28]{Coombes2014} for a related full-field inverse neural problem, where the authors investigated how to construct the kernel $\omega$ under the assumption that both the initial state~$a_0(x)$ and the solution~$a(x,t)$ are known for all~$x$ and~$t$. A general controllability question--the ability to steer the system from an initial state to a given target state--was engaged in \cite{tamekue2024mathematical}. The latter work has provided new insights into the ability of a constant-in-time input to steer the solution of \eqref{eq:NF-intro} from an initial state to a desired target state over a small time horizon. While \cite{tamekue2024mathematical} proposes synthesizing such inputs, the implicit dependence on an abstract operator, denoted $\cO(\tau)$, poses challenges for practical application, particularly as $\cO(\tau)$ depends on the state $a$ and the input $I$ even for a small time horizon.

From a neural control standpoint, a feedforward synthesis of a constant-in-time input that approximately solves the control objective is valuable, especially in modeling visual illusions, where transient dynamics may play a critical role. This is particularly relevant for understanding how specific sensory stimuli, modeled as constant-in-time inputs, influence cortical dynamics in transient regimes. Noninvasive neurostimulation techniques, such as transcranial Direct Current Stimulation (tDCS), can also benefit from constant-time input synthesis for small time horizons. In tDCS, weak constant electrical currents are applied to modulate brain activity \cite{bergmann2009acute}, and modeling these effects can help predict how sustained stimulation influences neural dynamics over short time intervals \cite[Section~11.10]{schiff2011neural}. Such an approach may be relevant for emerging interests in translational neuroscience, such as the use of tDCS for memory enhancement \cite{grover2022long} and other therapeutic applications.

The open-loop control synthesis with constant-in-time inputs can also be used to mechanistically steer the neural field equation~\eqref{eq:NF-intro} toward generating a desired stationary pattern over a short time horizon---even patterns that would not spontaneously emerge under the given parameter regime, such as when $\mu = \mu_c$, $\mu < \mu_c$, or $\mu > \mu_c$ (see~\eqref{eq:mu criticals} for the definition of the critical parameter $\mu_c$). In his seminal work~\cite{amari1977dynamics}, Amari analyzed conditions under which spatial patterns of neural activity---such as homogeneous active states, localized bumps, traveling waves, or stripes---can arise in the absence of external input, driven purely by spontaneous symmetry-breaking due to noise or initial heterogeneity.

It is well established that for sufficiently small values of the gain parameter $\mu$, specifically when $\mu < \mu_c$, only the trivial or homogeneous active state exists and is globally exponentially stable. However, when $\mu$ exceeds the critical threshold $\mu_c$, localized bump solutions can emerge through a reversible Hopf bifurcation with $1{:}1$ resonance~\cite{faye2013localized}; see also~\cite{coombes2005waves, curtu2004pattern}. As a result, in the absence of external input and starting from a homogeneous initial state $a_0$, the solution to~\eqref{eq:NF-intro} remains homogeneous for all $T > 0$ if $\mu < \mu_c$. Nevertheless, our numerical illustrations (in both one- and two-dimensional spatial domains) show that the proposed generic constant-in-time input syntheses can robustly drive the system from a homogeneous state to a localized bump state within an arbitrarily small time horizon, even in parameter regimes where such patterns would not arise spontaneously.

The remainder of the article is organized as follows.
	Section~\ref{not:general} introduces the general notations used throughout the article, together with the assumptions considered in the parameters in~\eqref{eq:NF-intro}.
	In Section~\ref{ss:Banach fixed-point theorem}, we begin by directly integrating~\eqref{eq:NF-intro} to derive an implicit control synthesis based on the Banach fixed-point theorem, and we highlight its limitations even in the linear setting.
	To overcome these drawbacks, Section~\ref{ss:generic synthesis} presents the generic synthesis results of the article. Section~\ref{sss:model-formulation} reformulates~\eqref{eq:NF-intro} in an abstract form, setting the stage for a more general control framework. Section~\ref{ss:representation of solutions} addresses the well-posedness of the dynamics by giving a solution representation that underpins the proposed synthesis, while Section~\ref{s:Generic syntheses in practice} discusses the practical feasibility of the proposed generic synthesis inputs. Section~\ref{ss:comments and open questions} provides remarks and outlines directions for further investigation. 	In Section~\ref{s:application in predicting visual illusions}, we discuss how the developed constant-in-time input syntheses can be leveraged to predict or induce visual illusions in a mechanistic modeling context.
	Numerical illustrations supporting the effectiveness of the proposed control inputs are reported in Section~\ref{s:NI}.
	The main proofs are collected in Section~\ref{s:proofs of main results}, and technical auxiliary results are deferred to Appendix~\ref{s:complement results}.

\begin{notation}\label{not:general}
    
Unless otherwise stated, $p$ will denote a real number satisfying $1\le p\le\infty$, and $q$ will denote the conjugate to $p$ given by $1/p+1/q = 1$. We adopt the convention that the conjugate of $p=1$ is $q = \infty$ and vice-versa. For $d\in\N^{*}$, we denote by $L^p(\R^d)$ the Lebesgue space of class of real-valued measurable functions $u$ on $\R^d$ such that $|u|$ is integrable over $\R^d$ if $p<\infty$, and $|u|$ is essentially bounded over $\R^d$ when $p=\infty$. We endow these spaces with their standard norms
$$\|u\|_p^p = \int_{\R^d}|u(x)|^p\,dx\quad\mbox{and}\quad \|u\|_\infty = \operatorname{ess}\sup_{x\in\R^d}|u(x)|.$$

We use $\mathscr{L}(X, Y)$ to denote the space of linear bounded operators between two normed vector spaces $X$ and $Y$, and simply $\mathscr{L}(X)$ when $X=Y$. If $A\in \mathscr{L}(X)$, then $e^A\in \mathscr{L}(X)$ and $\sigma(A)$ are, respectively, the exponential operator and the spectrum of the operator $A$ defined by
\begin{equation}\label{eq:exponential operator}
    e^A = \sum_{n=0}^\infty\frac{A^n}{n!} \qquad\text{and}\qquad\sigma(A)=\C\bs\rho(A)
\end{equation}
  where $\rho(A):=\{\lambda \in\C\mid A-\lambda \idty\in\mathscr{L}(X)\;\mbox{is one to one and onto}\}$.
  
 For $x\in\R^d$, we denote by $|x|$ its Euclidean norm, and the scalar product with $\xi\in\R^d$ is defined by $\langle x,\xi\rangle=x_1\xi_1+x_2\xi_2+\cdots +x_d\xi_d$. Let $\cS(\R^d)$ be the Schwartz space of rapidly-decreasing $C^\infty(\R^d)$ functions.  The Fourier transform of $u\in \cS(\R^d)$ is defined by
\begin{equation}\label{eq::Fourier transform in S}
	\widehat{u}(\xi):= \cF\{u\}(\xi)=\int_{\R^d} u(x)e^{-2\pi i\langle x,\xi\rangle}\,dx\qquad\qquad\forall\xi\in\R^d.
\end{equation}
The spatial convolution of two functions $u\in L^1(\R^d)$ and $v\in L^p(\R^d)$, $1\le p\le\infty$ is defined by
\begin{equation}\label{eq::spatial convolution}
	(u\ast v)(x) = \int_{\R^d}u(x-y)v(y)\,dy\qquad\qquad\forall x\in\R^d.
\end{equation}
\end{notation}

Finally, we recall that for every $w\in L^p(\R^d)$, $1\le p<\infty$, its dual $w^{*}:=w|w|^{p-2}\|w\|_p^{2-p}\in L^q(\R^d)$. Moreover, for $1<p<\infty$, the mapping $\varphi: L^p(\R^d)\to\R$, $w \mapsto \varphi(w) = \|w\|_p^2/2$ is differentiable and $D\varphi(w)\phi = \langle w^{*},\phi\rangle_{L^q, L^p}$, for all $w, \phi\in L^p(\R^d)$ where
\begin{equation}
    \langle w^{*},\phi\rangle_{L^q, L^p}:=\int_{\R^d}w^{*}(x)\phi(x)\,dx.
\end{equation}
For $p=2$, it coincides with the usual scalar product in $L^2(\R^d)$ that we denote as $\langle w^{*},\phi\rangle_{L^2}=\langle w,\phi\rangle_{L^2}$.

\begin{assumption}\label{ass:general assumption}

Unless otherwise stated, we assume that the transfer function $f$ belongs to the class $C^2(\R)$, is non-decreasing, $f'$, and $f''$ are bounded so that $f$ and $f'$ are globally Lipschitz on $\R$. For clarity in the presentation, we assume that $f(0) = 0$ and $\|f'\|_\infty= 1$. The latter are without loss of generality since we can replace\footnote{This change of functions yields \eqref{eq:NF-intro} with kernel $\tilde{\omega}=\omega/\lambda$, and input $\tilde{I}=-\mu f(0)\widehat{\omega}(0)+I/\lambda$.} $f$ by $\widetilde{f}(s) = f(\lambda s)-f(0)$ with $\lambda:= 1/\|f'\|_\infty$.

    We also assume\footnote{If one works in $L^p(\R^d)$ for a particular $1\le p\le\infty$, it is sufficient that $\omega\in L^1(\R^d)\cap L^p(\R^d)\cap L^q(\R^d)$. Indeed, $\omega\in L^1(\R^d)\cap L^p(\R^d)$ as highlighted in \cite[Remark~4.5]{tamekue2024mathematical}, and $\omega\in L^q(\R^d)$ as it is required for some of our main results in this article.} that the connectivity kernel $\omega\in \cS(\R^d)$, and we introduce the parameters
    \begin{equation}\label{eq:mu criticals}
        \mu_0:=\frac{\alpha}{\|\omega\|_1}\qquad\qquad
        \mu_1:=\frac{\alpha}{\|\widehat{\omega}\|_\infty},\qquad \mu_c:=\frac{\alpha}{\|\widehat{\omega}\|_\infty f'(0)}.
    \end{equation}
    that we will refer to when needed. Then, $\mu_0\le\mu_1\le\mu_c$, the parameter $\mu_0$ ($\mu_1$ if $p=2$) appears as the natural largest value of $\mu$ up to which \eqref{eq:NF-intro} admits a unique stationary state which is globally exponentially stable in the space $L^p(\R^d)$ (see, for instance, \cite{tamekue2022reproducing}), and when $d=2$, assuming $\mu<\mu_0$ ($\mu<\mu_1$ if $p=2$) enables the modeling of certain visual illusions \cite{bolelli2025neural,tamekue2023cortical,tamekue2024mathematical,tamekue2025reproducibility} using \eqref{eq:NF-intro}. Whereas the parameter $\mu_c$ is referred to as the bifurcation point of the homogeneous state (which corresponds to $a=0$ since it is assumed that $f(0)=0$) and corresponds to the value of $\mu$ where \eqref{eq:NF-intro} generates certain geometric hallucinatory patterns \cite{bressloff2001geometric,ermentrout1979mathematical} that spontaneously emerge in the primary visual cortex (V1 henceforth) or visual illusions \cite{nicks2021understanding} when a state-dependent sensory input drives V1's activity. See also \cite{veltz2010local} for the analysis of persistent states that emerge at this point.
\end{assumption}

\begin{remark}
    We do not explicitly require $f$ to be bounded. The assumption that $f$ is non-decreasing is not required for most of the results presented in this article. For instance, assuming $f'(0)>0$ yields $\mu_c>0$.
\end{remark}

\section{Main Results}\label{s:settings and main results}

This section presents the main theoretical contributions of the article. Section~\ref{ss:Banach fixed-point theorem} introduces a constructive control synthesis based on the Banach fixed-point theorem. Relying on standard assumptions for the transfer function and connectivity kernel, this method yields an iterative scheme for synthesizing constant-in-time inputs that steer the system from a given initial state to a prescribed target.

Although effective under appropriate conditions, this approach---illustrated through the linear case---reveals intrinsic limitations: it fails to fully characterize the controllability properties for constant inputs, especially in critical or supercritical regimes of the parameter~$\mu$. These shortcomings motivate the development of a more general and robust synthesis framework in Section~\ref{ss:generic synthesis}.

Before defining the notion of controllability we consider in this article, let us recall the following definition. See, for instance, \cite[Remark~4.2.3]{tucsnak2009observation}.
\begin{definition}\label{def:step function}
   A step function on $[0, T]$ with values in $L^p(\R^d)$ is a piecewise constant in time function defined on $[0, T]$. Constant in time functions $I\in L^p(\R^d)$ on $[0, T]$ are natural examples of such functions.
\end{definition}

\begin{definition}[Step controllability]\label{def:exact control notion}
Let $1\le p\le\infty$ and $T>0$. We say that \eqref{eq:NF-intro} is step controllable over the time interval $[0, T]$ if, for all $(a_0, a_1)\in L^p(\R^d)^2$, there exists a step function $I$ on $[0, T]$ with value in $L^p(\R^d)$ such that the solution of \eqref{eq:NF-intro} with $a(0)=a_0$ satisfies $a(T) = a_1$.     
\end{definition}

\subsection{Synthesis based on the Banach fixed-point theorem}\label{ss:Banach fixed-point theorem}

In this section---and only here---we assume that the transfer function $f$ is non-decreasing, globally Lipschitz continuous, and differentiable at $0$.
We also assume only that the connectivity kernel $\omega \in L^1(\R^d)$. Under these assumptions, the parameters $\mu_0$, $\mu_1$, and $\mu_c$ remain well-defined as given in \eqref{eq:mu criticals}. 

Within this setting, we begin with the following well-posedness result, the proof of which can be found in \cite{veltz2010local} or by invoking classical textbooks such as \cite[Chapter~6]{pazy2012semigroups}. 

\begin{proposition}\label{pro:well-posedness}
     Let $p\in[1, \infty]$ and $(a_0, I)\in L^p(\R^d)\times L^\infty([0, \infty); L^p(\R^d))$. There exists a unique solution $a_I\in C^0([0, \infty); L^p(\R^d))$ to \eqref{eq:NF-intro}. Letting $a_I(t):=a_I(\cdot,t)$, one has the following representation
     \begin{equation}\label{eq:implicit sol representation}
    a_I(t) = e^{-\alpha t}a_0+\mu\int_0^te^{-\alpha(t-s)}\omega\ast f(a_I(s))\,ds+\int_{0}^{t}e^{-\alpha(t-s)}I(s)ds,\qquad\forall t\ge 0.
\end{equation}
Moreover, the following estimate holds
     \begin{equation}\label{eq:a when p=finite}
       \|a_I(t)\|_p\le e^{-\alpha t(1-\frac{\mu}{\mu_0})}\|a_0\|_p+t\max(1,e^{-\alpha t(1-\frac{\mu}{\mu_0})})\sup_{0\le\tau\le t}\|I(\tau)\|_p,\qquad\forall t\ge 0.
\end{equation}
\end{proposition}

As is common in control synthesis approaches, determining a constant-in-time input that steers the system~\eqref{eq:NF-intro} from an initial state~$a_0$ to a target state~$a_1 \in L^p(\mathbb{R}^d)$ over a finite time horizon~$T > 0$ amounts to solving, for~$I \in L^p(\mathbb{R}^d)$, the terminal condition $a_I(T) = a_1$, where $a_I(\cdot)$ denotes the solution of~\eqref{eq:NF-intro} associated with the input~$I$. Substituting into the implicit solution representation~\eqref{eq:implicit sol representation}, the requirement $a_I(T) = a_1$ yields the condition that~$I \in L^p(\mathbb{R}^d)$ must satisfy
\begin{equation}\label{eq:I from BFT}
    I = \alpha(1 - e^{-\alpha T})^{-1} \left\{ a_1 - e^{-\alpha T} a_0 - \mu \int_0^T e^{-\alpha(T - t)} \, \omega \ast f(a_I(t)) \, dt \right\}.
\end{equation}

Equation~\eqref{eq:I from BFT} is implicit w.r.t. the unknown~$I \in L^p(\mathbb{R}^d)$, and the key challenge lies in identifying conditions under which it admits at least one solution that is numerically implementable. This motivates the main result of this section, which relies on the Banach fixed-point theorem. The proof is outlined in Section~\ref{ss:proof of the Banach fixed-point theorem}.

\begin{theorem}\label{thm:Banach fixed-point theorem}
    Let $p\in[1, \infty]$ and $(a_0, a_1)\in L^p(\R^d)^2$. Define the sequence $(I_n)_n\subset L^p(\R^d)$ by
    \begin{equation}
        \begin{cases}
            I_0\in L^p(\R^d)\\
            I_{n+1} = \displaystyle\alpha(1-e^{-\alpha t})^{-1}\left\{a_1-e^{-\alpha T}a_0-\mu\int_0^Te^{-\alpha(T-t)}\omega\ast f(a_{I_n}(t))\,dt\right\},\qquad n\ge 0\\
            a_{I_n}\in L^p(\R^d)\;\text{solution of}\quad \dot{a}(t)  = -\alpha a(t) + \mu\omega\ast f(a(t)) + I_n,\quad a(0) = a_0.
        \end{cases}
    \end{equation}
    Assume that one of the following conditions holds 
    \begin{enumerate}
    \item $T>0$ and $\mu<\mu_0/2$;
        \item $T<T_*(\frac{\mu}{\mu_0})$ and $\mu\ge\mu_0$ where $T_*(\frac{\mu}{\mu_0})>0$ is the unique solution of $L(T)=1$ with
        \begin{equation}\label{eq:lipschitz constant of Phi}
            L(T):=\begin{cases}\displaystyle\frac{\alpha T}{1-e^{-\alpha T}}-1
            \quad&\text{if}\quad \mu=\mu_0\\
            \displaystyle\frac{1}{r-1}\left(\frac{e^{\alpha rT}-1}{e^{\alpha T}-1}-r\right)\quad&\text{if}\quad r:=\mu/\mu_0\neq 1.
            \end{cases}
        \end{equation}
    \end{enumerate}
    Then, $(I_n)_n$ converges in $L^p(\R^d)$ to $I\in L^p(\R^d)$ given by \eqref{eq:I from BFT} and the corresponding solution $a_I(\cdot)$ to \eqref{eq:NF-intro} satisfies $a_I(T)=a_1$. 
    
\end{theorem}

\begin{remark}
    In Proposition~\ref{pro:well-posedness} and Theorem~\ref{thm:Banach fixed-point theorem}, one can replace $\mu_0$ by $\mu_1$ in the case of $p=2$ to get sharp results. This is justified via a Fourier transform argument using Parseval's identity and Plancherel's theorem. See, for instance, the proof of Lemma~\ref{lem:norm of exp A_t(U_t(a0))}.
\end{remark}

\begin{remark}
Theorem~\ref{thm:Banach fixed-point theorem} provides an implementable iterative synthesis of a constant in time input that steers \eqref{eq:NF-intro} from a given initial state $a_0 \in L^p(\R^d)$ to any target state $a_1 \in L^p(\R^d)$ over a finite time horizon $ T > 0 $. However, the result comes with the following limitations
\begin{enumerate}
    \item It covers the \emph{subcritical regime} $\mu < \mu_0$ for all $T>0$ if $\mu < \mu_0/2$;
    \item It covers the \textit{critical} and the \emph{supercritical regime} $\mu \ge \mu_0 $ if $T < T_*(\mu/\mu_0)$, where $T_*(\cdot) > 0$ decreases as $\mu/\mu_0 $ increases.
\end{enumerate}

We can also highlight the limitations of Theorem~\ref{thm:Banach fixed-point theorem} by comparing it with the linear case, as stated in the following proposition. The proof is straightforward and can be done by mimicking~\cite[Proposition III.2]{tamekue2024control}. 
\end{remark}
\begin{proposition}\label{pro:linear case}
    Let $p\in[1, \infty]$, $T>0$ and $a_0\in L^p(\R^d)$. Consider the linear dynamics
    \begin{equation}\label{eq:linear dynamics}
        \dot{a}(t)=Aa(t)+I,\quad a(0)=a_0,\qquad t\in [0,T]
    \end{equation}
    where $Au = -\alpha u+\mu\omega\ast u$, $u\in L^p(\R^d)$. Assume that the following spectral condition holds
        \begin{equation}\label{eq:sc linear}
            \sigma(A)\cap\left\{i\frac{2 \pi \ell}{T} \in \mathbb{C} \mid\ell\in\Z\right\}=\emptyset.
        \end{equation}
        Then, the solution $a_I\in C^\infty([0, T]; L^p(\R^d))$ of \eqref{eq:linear dynamics} satisfies $a_I(T)=a_1\in L^p(\R^d)$, if and only if, the constant in time input $I\in L^p(\R^d)$ is given by
        \begin{equation}\label{eq:input linear}
            I = (e^{TA}-\operatorname{Id})^{-1}A(a_1-e^{TA}a_0).
        \end{equation}
\end{proposition}
\begin{lemma}\label{lem:on the SA in the linear case}
    The spectral condition in~\eqref{eq:sc linear} is automatically satisfied if $\mu < \mu_0$, or, in the specific case where $p = 2$, if $\mu < \mu_1$. In particular, under these conditions, $I \in L^p(\mathbb{R}^d)$ in \eqref{eq:input linear} is well-defined.
\end{lemma}
\begin{proof}
    Let $t \in \mathbb{R}_+$. One proves as in Lemma~\ref{lem:norm of exp A_t(U_t(a0))} that
    \begin{equation}\label{eq:norm of exp tA}
        \|e^{tA}\|_{\mathscr{L}(L^p(\R^d))}\begin{cases}
            \le e^{-\alpha\left(1-\frac{\mu}{\mu_0}\right)t}&\quad\mbox{if}\quad 1\le p\le\infty\\
            \le e^{-\alpha\left(1-\frac{\mu}{\mu_1}\right)t}&\quad\mbox{if}\quad p=2.
        \end{cases}
    \end{equation}
    It follows that $\|e^{tA}\|_{\mathscr{L}(L^p(\mathbb{R}^d))} < 1$ for every $t > 0$ if $\mu < \mu_0$ (or $\mu<\mu_1$ when $p=2$). The first part of the lemma follows, and the second part then follows from the Neumann series expansion.
\end{proof}

Proposition~\ref{pro:linear case} ensures that in the linear case (i.e., when $f(s) = s$), a constant-in-time input that steers \eqref{eq:NF-intro} from an initial state $a_0$ to a target state $a_1$ over a finite time horizon $T > 0$ can be constructed explicitly in a feedforward manner. Notably, this input is well-defined only under the spectral condition specified in~\eqref{eq:sc linear}, which is less restrictive than those required in Theorem~\ref{thm:Banach fixed-point theorem}. 

The objective of the next section is to properly generalize this result to the fully nonlinear setting.

\subsection{Generic syntheses: Proper generalization of the linear case}\label{ss:generic synthesis}

This section introduces generic synthesis formulations that extend the linear case in a principled manner. Section~\ref{sss:model-formulation} recasts the system~\eqref{eq:NF-intro} within an abstract functional framework tailored to control-theoretic analysis and synthesis. In contrast to the fixed-point approach of Section~\ref{ss:Banach fixed-point theorem}, which requires only mild assumptions on~$f$, this formulation imposes stronger smoothness conditions (see Assumptions~\ref{ass:general assumption}), which are crucial for exploiting the flow maps of the drift dynamics in constructing a generic synthesis theory.

\subsubsection{Abstract Formulation}\label{sss:model-formulation}

For our purpose of extending the linear controllability results presented in Proposition~\ref{pro:linear case} in the fully nonlinear case, we recast \eqref{eq:NF-intro} as follows
\begin{equation}\label{eq::nonlinear control flow notation}
    \dot{a}(t) = N(a(t))+I,\qquad a(0) = a_0
\end{equation}
where for any $1\le p\le\infty$, the nonlinear map $N$ is given by
\begin{equation}\label{eq::operator N}
	N(u) = -\alpha u+\mu\omega\ast f(u),\qquad u\in L^p(\R^d).
\end{equation}
Equation~\eqref{eq::nonlinear control flow notation} defines an ordinary differential equation in the Banach space $L^p(\R^d) $ associated with $N$. Recall from \cite[Lemma~3.2]{tamekue2024mathematical} that for every $1\le p\le\infty$, $N$ is (globally) Lipschitz continuous from $L^p(\R^d) $ to itself. Therefore, one can define a (nonlinear) flow of operators $\{U_t\mid t\in\R\}$ in $L^p(\R^d)$ generated by $N$ such that for any $a_0\in L^p(\R^d)$, $U_t(a_0)$ is the solution to \eqref{eq::nonlinear control flow notation} when $I\equiv0$, namely
\begin{equation}\label{eq::nonlinear flow}
	\frac{d}{d t} U_t(a_0) = N(U_t(a_0)),\qquad\qquad
	U_0(a_0) = a_0.
\end{equation}
Denote by $V_t$ the inverse of $U_t$ for $t\in\R$, i.e., $V_{t}(a_0)=U_{-t}(a_0)$ for all $a_0\in L^p(\R^d)$. Then, it holds
\begin{equation}\label{eq::nonlinear flow psi}
	\frac{d}{d t}V_t(a_0) = -N(V_t(a_0)),\qquad
	V_0(a_0) = a_0.
\end{equation}
By Assumptions~\ref{ass:general assumption}, it follows, for example, from\footnote{For a detailed discussion of the notions of Fréchet and Gâteaux derivatives, see, for example, \cite[Chapter~7]{ciarlet2013linear}.} \cite[Lemma~B.8]{tamekue2024mathematical} that for every $1<p\le\infty$, $N\in C^1(L^p(\R^d); L^p(\R^d))$ (Gâteaux differentiable when $p=1$) and the operator norm of the Fréchet derivative $DN(\psi)$ is uniformly bounded w.r.t. $\psi\in L^p(\R^d)$, viz.
\begin{equation}\label{eq:operator norm of DN}
    \|DN(\psi)\|_{\mathscr{L}(L^p(\R^d))}\begin{cases}
        \le \alpha+\mu\|\omega\|_1&\quad\mbox{if}\quad 1<p\le\infty\\
        \le \alpha+\mu\|\widehat{\omega}\|_\infty&\quad\mbox{if}\quad p=2
    \end{cases}\qquad\qquad\forall\psi\in L^p(\R^d).
\end{equation}
Here the differential $DN(\psi)\in\mathscr{L}(L^p(\R^d))$ at every $\psi\in L^p(\R^d)$ is given by (see, \cite[Lemma~B.8]{tamekue2024mathematical})
\begin{equation}\label{eq:derivative of the drift}
    DN(\psi)\phi=-\alpha\phi+\mu\omega\ast[f'(\psi)\phi]\qquad\qquad\forall\phi\in L^p(\R^d).
\end{equation}
In particular, for any fixed $t\in\R$, $U_t$ and $V_t$ are Fréchet differentiable for every $a_0\in L^p(\R^d)$. If we use $DU_t(a_0)$ and $DV_t(a_0)$ to denote, respectively, the evaluation of these Fréchet derivatives at any $a_0\in L^p(\R^d)$, then they solve, respectively, the following 
\begin{equation}\label{eq:equation derivative of the flow}
    \begin{cases}
        \displaystyle\frac{d}{d t} DU_t(a_0)&=DN(U_t(a_0))DU_t(a_0)\\
        DU_0(a_0)&=\idty,
    \end{cases}\qquad\quad\begin{cases}
        \displaystyle\frac{d}{d t} DV_t(a_0)&=-DN(V_t(a_0))DV_t(a_0)\\
        DV_0(a_0)&=\idty.
    \end{cases}
\end{equation}
Moreover, it follows from\footnote{The result in \cite[Lemma~B.10]{tamekue2024mathematical} proved the invertibility of $DU_t(a_0)$ for every $t\ge 0$. However, this can be straightforwardly extended to $t\in\R$ since $U_t$ is well-defined for all $t\in\R$.} \cite[Lemma~B.10]{tamekue2024mathematical} that  $DU_t(a_0)$ is a well-defined invertible operator, and 
\begin{equation}\label{inverse of the differential of the flow}
    \left[DU_t(a_0)\right]^{-1} = DV_t(U_t(a_0)),\qquad t\in\R,\, a_0\in L^p(\R^d).
\end{equation}

\subsubsection{Representations and series expansions of the solution}\label{ss:representation of solutions}

This section provides dual representations of the solution to~\eqref{eq:NF-intro} and its series expansions based on the abstract formulation of Section~\ref{sss:model-formulation}. These representations emerge from the chronological calculus–inspired framework introduced by Agrachev and Gamkrelidze in the late 1970s \cite{agrachev1979exponential} to represent the solutions of nonautonomous ordinary differential equations on finite-dimensional manifolds. 

To the best of the authors' knowledge, this representation framework was first used in~\cite{tamekue2024mathematical} in the context of~\eqref{eq:NF-intro} and later instantiated in finite-dimensional settings in~\cite{tamekue2024control} to enable synthesis in continuous-time Hopfield-type recurrent neural networks with underactuated constant inputs. 

For completeness, we present these representations in the general case of time-varying inputs.
We start with the first solution representation that can be viewed as a natural extension of the solution representation in the linear case (Duhamel’s formula~\cite[p.~105]{pazy2012semigroups}), particularly at the final time horizon $T$, with the nonlinear flow and its differential replacing the exponential operator. In particular, it can be termed as a backward representation of the solution to \eqref{eq::nonlinear control flow notation}. The proof of the following is given in Section~\ref{ss:proof of backward representation}.

\begin{theorem}\label{thm:backward representation} 
    Let $p\in(1,\infty]$ and $T>0$. For all $(a_0, I)\in L^p(\R^d)\times L^1((0, T); L^p(\R^d))$, the solution $a\in C^0([0, T]; L^p(\R^d))$ to~\eqref{eq::nonlinear control flow notation} can be expressed as
    \begin{equation}\label{eq:backward representation}
        a(t) = U_{t-T}\left(U_T(a_0)+\int_{0}^{t}DU_{T-s}(a(s)) I(s)\,ds\right),\qquad\forall t\in[0, T].
    \end{equation} 
\end{theorem}

The second solution representation, termed as a forward representation of the solution to \eqref{eq::nonlinear control flow notation}, is as follows. The proof is identical to that of Theorem~\ref{thm:backward representation} and will be omitted for brevity.

\begin{theorem}\label{thm:forward representation} 
    Let $p\in(1,\infty]$ and $T>0$. For all $(a_0, I)\in L^p(\R^d)\times L^1((0, T); L^p(\R^d))$, the solution $a\in C^0([0, T]; L^p(\R^d))$ to~\eqref{eq::nonlinear control flow notation} can be expressed as
    \begin{equation}\label{eq:forward representation}
        a(t) = U_t\left(a_0+\int_{0}^{t}DV_s(a(s)) I(s)ds\right),\qquad\forall t\in[0, T].
    \end{equation} 
\end{theorem}

\begin{remark}\label{rmk:on L^1}
    Theorems~\ref{thm:backward representation} and~\ref{thm:forward representation} are stated for $p \in (1, \infty]$, thereby excluding the case $p = 1$. This restriction stems from the fact that the (Fréchet) derivative operators $DU_t(a_0)$ and $DV_t(a_0)$, defined for $a_0 \in L^p(\mathbb{R}^d)$, are well-defined only when $p \in (1, \infty]$. This follows from \cite[Lemma~B.8]{tamekue2024mathematical}, which establishes that the nonlinear map $N$ is Fréchet differentiable only in this range of $p$.

    Whether $N$ is Fréchet differentiable in $L^1(\mathbb{R}^d)$ remains unclear. The proof of Fréchet differentiability in \cite[Lemma~B.8]{tamekue2024mathematical} relies on a classical result (see, e.g., \cite[Proposition~2.8, p.~60]{tapia1971differentiation}): let $\cX$ and $\cY$ be normed vector spaces, and let $g: \cX \to \cY$. If $g$ is Gâteaux differentiable and its Gâteaux derivative $\partial g: \cX \to \mathscr{L}(\cX, \cY)$ is continuous, then $g$ is Fréchet differentiable and belongs to $C^1(\cX, \cY)$.

    In \cite[Lemma~B.8]{tamekue2024mathematical}, the authors showed that the Gâteaux derivative of $N$ is not continuous from $L^1(\mathbb{R}^d)$ to $\mathscr{L}(L^1(\mathbb{R}^d))$. While this does not rule out Fréchet differentiability in $L^1(\mathbb{R}^d)$, it implies that if $N$ were Fréchet differentiable in $L^1(\R^d)$, the map $DN: L^1(\mathbb{R}^d) \to \mathscr{L}(L^1(\mathbb{R}^d))$ could not be continuous. In other words, $N$ could not belong to $C^1(L^1(\mathbb{R}^d); L^1(\mathbb{R}^d))$.
\end{remark}

Building on the solution representation in Theorem~\ref{thm:backward representation}, the following proposition outlines the expansion strategy that forms the foundation of our subsequent analysis and input syntheses. The proof is provided in Section~\ref{ss:proof of expansion for forward nominal-state synthesis}.

\begin{proposition}\label{pro:expansion for forward nominal-state synthesis}
   Let $p\in[2,\infty]$, $\tau>0$ and  $T\ge\tau$. Let $a_0\in L^p(\R^d)$, and $I_{sf, \tau}\in L^1((0,T); L^p(\R^d))$ be the step function defined by
   \begin{equation}\label{eq:initial step function}
       I_{sf, \tau}(t)=\begin{cases}
           0&\quad\text{if}\quad 0\le t\le T-\tau\\
           I&\quad\text{if}\quad T-\tau< t\le T
       \end{cases}
   \end{equation}
   where $I\in L^p(\R^d)$. 
   Then, the solution $a_\tau(\cdot)$ to \eqref{eq::nonlinear control flow notation} corresponding to $I_{sf, \tau}$ expands at time $t=T$ as
    \begin{equation}\tag{F1}\label{eq:expansion for forward nominal-state synthesis}
        a_\tau(T) = U_T(a_0)-\sum_{n=0}^{\infty}\frac{(-\tau)^{n+1}DN\left(U_T(a_0)\right)^n}{(n+1)!}DU_\tau(U_{T-\tau}(a_0))I-\varphi_{\tau,T}(I)I.
    \end{equation}
  Here $\varphi_{\tau,T}(I):=\kappa_{\tau,T}(I)+\chi_{\tau,T}(I)$ with
    \begin{equation}\label{eq:kappa and chi}
        \kappa_{\tau,T}(I) = \sum_{n=1}^{\infty}\int_{0}^{\tau}\frac{(-1)^nt^{n+1}}{(n+1)!}\left[\frac{d}{dt}Z_{t,\tau,T}^n\right]Q_{t,\tau,T}\,dt,\qquad\chi_{\tau,T}(I)= \sum_{n=1}^{\infty}\int_{0}^{\tau}\frac{(-t)^nZ_{t,\tau,T}^{n-1}}{n!}P_{t,\tau,T}\,dt
    \end{equation}
    where $Q_{t,\tau,T}:=DU_t(a_\tau(T-t))$, $Z_{t,\tau,T}:=DN(U_t(a_\tau(T-t)))$ and $P_{t,\tau,T}:=\partial DU_t(a_\tau(T-t))\dot{a}_\tau(T-t)$ is the Gâteaux derivative of $DU_t$ at $a_\tau(T-t)$ in the direction of $\dot{a}_\tau(T-t)$.
\end{proposition}

\begin{remark}
   In Proposition~\ref{pro:expansion for forward nominal-state synthesis}, the expansion of the solution $a(\cdot)$ to~\eqref{eq::nonlinear control flow notation} on $[0,T]$ for a constant in time input $I \in L^p(\mathbb{R}^d)$ is obtained by setting $\tau = T$.
\end{remark}
\begin{remark}\label{rmk:on the range of p}
    Proposition~\ref{pro:expansion for forward nominal-state synthesis} is stated for $p \in [2, \infty]$, excluding the case $p \in (1, 2)$. While the operators $DN\left(U_T(a_0)\right)$ and $DN(U_t(a_\tau(T-t)))$ belong to $\mathscr{L}(L^p(\mathbb{R}^d))$ for all $p \in (1, \infty]$, it follows from Proposition~\ref{pro:second differential of N} that the operators $\kappa_{\tau,T}(I)$ and $\chi_{\tau,T}(I)$ are well-defined and belong to $\mathscr{L}(L^p(\mathbb{R}^d))$ when $p \in [2, \infty]$.
\end{remark}

\begin{remark}\label{rmk:SES}
    The solution of~\eqref{eq::nonlinear control flow notation} is implicitly represented in both Theorems~\ref{thm:backward representation} and~\ref{thm:forward representation}. Each representation enables multiple expansion strategies for expressing the final state as a series involving either the initial state $a_0$ or the final state itself. 
    
    \noindent \textbf{From the backward representation} (Theorem~\ref{thm:backward representation}):
    \[
    a_\tau(T) = U_T(a_0) + \int_0^T DU_{T-t}(a_\tau(t)) I_{sf, \tau}(t) \, dt=U_T(a_0) + \int_0^\tau DU_t(a_\tau(T-t)) I \, dt,
    \]
    two expansion strategies are possible.
    \begin{itemize}
        \item[(i)] Using the identity $\int_0^\tau DU_t(a_\tau(T-t)) I \, dt = \int_0^\tau t' DU_t(a_\tau(T-t)) I \, dt$, and applying successive integration by parts, we obtain \eqref{eq:expansion for forward nominal-state synthesis}. This yields the \emph{forward nominal-state} synthesis.
        \item[(ii)] Alternatively, using $\int_0^\tau DU_t(a_\tau(T-t)) I \, dt = \int_0^\tau (t-\tau)' DU_T(a_\tau(T-t)) I \, dt$, we derive
        \begin{equation}\tag{F2}\label{eq:forward2}
            a_\tau(T) = U_T(a_0) + \sum_{n=0}^\infty \frac{\tau^{n+1} DN(a_\tau(T))^n}{(n+1)!} I - \beta_{\tau,T}(I) I
        \end{equation}
        for some operator $\beta_{\tau,T}(I) \in \mathscr{L}(L^p(\mathbb{R}^d))$. This leads to the \emph{forward final-state} synthesis.
    \end{itemize}

    \medskip
    \noindent \textbf{From the forward representation} (Theorem~\ref{thm:forward representation}):
    \[
    V_T(a_\tau(T)) = a_0 + \int_0^T DV_t(a_\tau(t)) I_{sf, \tau}(t) \, dt = a_0 + \int_0^\tau DV_{T-t}(a_\tau(T-t)) I \, dt,
    \]
    two expansion strategies arise.
    \begin{itemize}
        \item[(i)] Using $\int_0^\tau DV_{T-t}(a_\tau(T-t)) I \, dt  = \int_0^\tau t' DV_{T-t}(a_\tau(T-t)) I \, dt$, we obtain
        \begin{equation}\tag{B1}\label{eq:backward1}
            V_T(a_\tau(T)) =  a_0 - \sum_{n=0}^\infty\frac{(-\tau)^{n+1} DN(a_0)^n}{(n+1)!}DV_{T-\tau}(U_{T-\tau}(a_0)) I - \eta_{\tau, T}(I) I
        \end{equation}
        for some operator $\eta_{\tau, T}(I) \in \mathscr{L}(L^p(\mathbb{R}^d))$. This gives rise to the \emph{backward initial-state} synthesis.
        
        \item[(ii)] Alternatively, using $\int_0^\tau DV_{T-t}(a_\tau(T-t)) I \, dt = \int_0^\tau(t-\tau)' DV_{T-t}(a_\tau(T-t)) I \, dt$, we obtain
        \begin{equation}\tag{B2}\label{eq:backward2}
            V_T(a_\tau(T)) = a_0 + \sum_{n=0}^\infty \frac{\tau^{n+1} DN(V_T(a_\tau(T)))^n}{(n+1)!} DV_T(a_\tau(T)) I - \zeta_{\tau, T}(I) I,
        \end{equation}
        for some operator $\zeta_{\tau, T}(I) \in \mathscr{L}(L^p(\mathbb{R}^d))$. This yields the \emph{backward nominal-state} synthesis.
    \end{itemize}
\end{remark}

\subsubsection{Forward nominal-state  synthesis}\label{sss:implicit forward nominal-state synthesis}

This section presents the generic synthesis based on the series expansion of the solution to~\eqref{eq::nonlinear control flow notation} given in Proposition~\ref{pro:expansion for forward nominal-state synthesis}. We provide a detailed analysis of this approach, which serves as a prototypical example of the family of synthesis strategies derived from the dual representations of the solution of~\eqref{eq::nonlinear control flow notation}, as discussed in Remark~\ref{rmk:SES}. The remaining synthesis strategies---the forward final-state, backward initial-state, and backward nominal-state formulations---can be derived through analogous arguments.

The first results are presented in the following theorem. It provides an implicit step-function input that solves the control objective. The proof is presented in Section~\ref{ss:proof of forward nominal-state synthesis}.

\begin{theorem}\label{thm:forward nominal-state synthesis}
      Let $p\in [2, \infty]$ and $(a_0,\, a_1) \in L^p(\mathbb{R}^d)^2$. Let $T\ge\tau$ where $\tau>0$ be such that the following spectral condition holds
    \begin{equation}\label{eq:SC forward nominal-state synthesis}
         \sigma(DN(U_T(a_0)))\cap\left\{i \frac{2 \pi \ell}{\tau} \in \mathbb{C} \mid\ell\in\Z\right\}=\emptyset.
    \end{equation}
     Then, the solution $a_\tau(\cdot)$ to~\eqref{eq::nonlinear control flow notation} corresponding to the step-function input $I_{sf, \tau}(\cdot)$ defined in \eqref{eq:initial step function} satisfies $a_\tau(T)=a_1$, if and only if, $I\in L^p (\R^d)$ satisfies
    \begin{equation}\label{eq:implicit forward nominal-state synthesis}
    \left[\idty-\cA_{\tau,T}(a_0)\varphi_{\tau,T}(I)\right]I=\cA_{\tau,T}(a_0)(a_1-U_T(a_0))
\end{equation}
where $\varphi_{\tau,T}(I)\in\mathscr{L}(L^p(\R^d))$ is defined as in Proposition~\ref{pro:expansion for forward nominal-state synthesis} and $\cA_{\tau,T}(a_0)\in\mathscr{L}(L^p(\R^d))$ is given by
\begin{equation}\label{eq:operator B_T}
        \cA_{\tau,T}(a_0) := DV_\tau(U_T(a_0))\left[\idty-e^{-\tau DN(U_T(a_0))} \right]^{-1}DN(U_T(a_0)).
    \end{equation}
\end{theorem}

\begin{remark}\label{rmk:on the SA in the nonlinear settings}
We first refer the reader to Remark~\ref{rmk:on the range of p} for a discussion of the restriction to the range $p \in [2, \infty]$. As in the linear case (see Lemma~\ref{lem:on the SA in the linear case}), the spectral condition in~\eqref{eq:SC forward nominal-state synthesis} is automatically satisfied when $\mu < \mu_0$---or even $\mu < \mu_1$ when $p = 2$---as a consequence of Lemma~\ref{lem:norm of exp A_t(U_t(a0))}. 

By contrast, in the nonlinear setting, it remains unclear whether the implicit equation~\eqref{eq:implicit forward nominal-state synthesis}, which defines the candidate constant-in-time input, admits at least one solution $I \in L^p(\mathbb{R}^d)$ when the operator $\varphi_{\tau,T}(I)$ is not identically zero. Nevertheless, \eqref{eq:kappa and chi} shows that $\varphi_{\tau,T}(I) \equiv 0$ when the transfer function $f$ is linear.
\end{remark}

The step-function input synthesis in Theorem~\ref{thm:forward nominal-state synthesis} is implicitly defined in terms of the unknown $I \in L^p(\mathbb{R}^d)$, which may limit the ability of its direct use in practice or numerical implementation. To address this, we now derive an explicit, tractable, and effective approximation of the step-function input synthesized in Theorem~\ref{thm:forward nominal-state synthesis}. This approximation offers practical and computational advantages, particularly for numerical implementations.

We begin with the following result, which provides an expansion of $\varphi_{\tau,T}(I)$ as $\tau \to 0$. The proof is provided in Section~\ref{ss:proof of expansion of varphi_tau,T}.

\begin{proposition}\label{pro:expansion of varphi_tau,T}
   Under the assumptions of Theorem~\ref{thm:forward nominal-state synthesis}, the following expansion holds
     \begin{equation}\label{eq:expansion of varphi_tau,T}
         \varphi_{\tau, T}(I) = \left[e^{\tau \Phi_T(a_0)}-\idty\right]\Phi_T(a_0)^{-1}\left[e^{-\tau \Phi_T(a_0)}DU_\tau(U_{T-\tau}(a_0))-\idty\right]+\cO(\tau^3),\quad\text{as}\quad\tau\to 0.
     \end{equation}
     where $\Phi_T(a_0):=DN(U_T(a_0))$. Moreover, the following expansion also holds
     \begin{equation}\label{eq:A_tau,T and varphi_tau,T}
         \idty-\cA_{\tau,T}(a_0)\varphi_{\tau,T}(I)=DV_\tau(U_T(a_0))e^{-\tau \Phi_T(a_0)}+\cO(\tau^2),\quad\text{as}\quad\tau\to 0.
     \end{equation}
\end{proposition}

The following result provides an explicit input synthesis along with an estimate of the endpoint error.

\begin{theorem}\label{thm:forward nominal-state synthesis step function}
   Let $p\in [2, \infty]$ and $(a_0,\, a_1) \in L^p(\mathbb{R}^d)^2$. Let $T\ge\tau$ where $\tau>0$ be such that the spectral condition~\eqref{eq:SC forward nominal-state synthesis} is satisfied. Define $I_{fn}\in L^p(\R^d)$ by
   \begin{equation}\label{eq:approximation of the forward nominal-state synthesis input}
       I_{fn} = \left[e^{\tau DN(U_T(a_0))}-\idty\right]^{-1}DN(U_T(a_0))(a_1-U_T(a_0)).
   \end{equation}
   Then, the solution $a_\tau(\cdot)$ to~\eqref{eq::nonlinear control flow notation} corresponding to the step-function input $\tilde{I}_{sf,\tau}:[0, T]\to L^p(\R^d)$,
   \begin{equation}
       \tilde{I}_{sf,\tau}(t)=\begin{cases}
           0&\quad\text{if}\quad 0\le t\le T-\tau\\
           I_{fn}&\quad\text{if}\quad T-\tau< t\le T
       \end{cases}
   \end{equation}
   satisfies the following endpoint error estimates
   \begin{equation}
       \|a_\tau(T)-a_1\|_p=\cO(\tau^2),\quad\text{as}\quad\tau\to 0.
   \end{equation}
\end{theorem}

\begin{remark}
We emphasize that a constant-in-time input candidate on $[0, T]$ corresponds to the special case $T = \tau$ in Theorem~\ref{thm:forward nominal-state synthesis step function}. The notation $I_{fn}$ in~\eqref{eq:approximation of the forward nominal-state synthesis input} stands for \emph{forward nominal} synthesis, where ``nominal'' refers to the evaluation of $DN(\cdot)$ along the nominal trajectory $U_t(a_0)|_{t = T}$ of~\eqref{eq::nonlinear control flow notation}.
\end{remark}

\begin{proof}[\textit{Proof} of Theorem~\ref{thm:forward nominal-state synthesis step function}]
    Using~\eqref{eq:expansion for forward nominal-state synthesis}, \eqref{eq:expansion of varphi_tau,T}, and the definition of $I_{fn}$ from~\eqref{eq:approximation of the forward nominal-state synthesis input}, we obtain
    \begin{eqnarray}\label{eq:solution proof}
        a_\tau(T)&=&  U_T(a_0)-\sum_{n=0}^{\infty}\frac{(-\tau)^{n+1}\Phi_T(a_0)^n}{(n+1)!}DU_\tau(U_{T-\tau}(a_0))I_{fn}-\varphi_{\tau,T}(I_{fn})I_{fn}\nonumber\\
        &\underset{\tau \sim 0}{=}&\left[e^{\tau \Phi_T(a_0)}-\idty\right]\Phi_T(a_0)^{-1}\left[e^{\tau \Phi_T(a_0)}-\idty\right]^{-1}\Phi_T(a_0)(a_1-U_T(a_0))+\cO(\tau^2)
        \underset{\tau \sim 0}{=}a_1+\cO(\tau^2)
    \end{eqnarray}
    since $\Phi_T(a_0)^{-1}$ commutes with $e^{\pm\tau \Phi_T(a_0)}$, where $\Phi_T(a_0):=DN(U_T(a_0))$. Moreover, $\cO(\tau^3)$ term in~\eqref{eq:expansion of varphi_tau,T} becomes $\cO(\tau^2)$ in~\eqref{eq:solution proof} due to $
\left[e^{\tau \Phi_T(a_0)}-\idty\right]^{-1}\, \cO(\tau^3) \underset{\tau \sim 0}{=} \cO(\tau^2)$,
since $\Phi_T(a_0)$ is uniformly bounded with respect to $T$ and $a_0$ so that $\left[e^{\tau \Phi_T(a_0)}-\idty\right]^{-1}\underset{\tau \sim 0}{=}\cO(\tau^{-1})$.
\end{proof}


\begin{remark}
    Observe that \eqref{eq:A_tau,T and varphi_tau,T} ensures that \eqref{eq:implicit forward nominal-state synthesis} admits the solution $I\in L^p(\R^d)$ given by
    \begin{eqnarray*}
I&=&e^{\tau\Phi_T(a_0)}\left[DV_\tau(U_T(a_0))\right]^{-1}\cA_{\tau,T}(a_0)\left(a_1-U_T(a_0)\right)+\cO(\tau^2)\qquad\text{as}\qquad\tau\to 0\nonumber\\
        &=&I_{fn}+\cO(\tau^2)\qquad\text{as}\qquad\tau\to 0
    \end{eqnarray*}
    where $I_{fn}\in L^p(\R^d)$ is given by \eqref{eq:approximation of the forward nominal-state synthesis input}.
\end{remark}

\section{Generic syntheses in practice}\label{s:Generic syntheses in practice}

Building on the implicit forward nominal-state synthesis and its explicit approximation developed in Section~\ref{sss:implicit forward nominal-state synthesis}, this section extends the analysis to the forward final-state, backward initial-state, and backward nominal-state formulations. Section~\ref{ss:On the spectral conditions} then discusses the spectral conditions needed for these syntheses and shows that they are naturally satisfied in many practical scenarios.

From here on, we focus exclusively on constant-in-time input syntheses, corresponding to the specific case $T = \tau > 0$ as defined in Proposition~\ref{pro:expansion for forward nominal-state synthesis} and Section~\ref{sss:implicit forward nominal-state synthesis}.

\subsection{Explicit approximations for the remaining syntheses}\label{ss:Complementarity of the zero-order approximation}

This section provides explicit approximations for the forward final-state, backward initial-state, and backward nominal-state syntheses, derived from the solution expansions in~\eqref{eq:forward2}, \eqref{eq:backward1}, and~\eqref{eq:backward2}. 
The next proposition summarizes these three syntheses. The proof follows the same lines as Theorem~\ref{thm:forward nominal-state synthesis step function} and is omitted for brevity.

\begin{proposition}\label{pro:the three other syntheses}
    Let $p\in[2, \infty]$, $(a_0,\, a_1) \in L^p(\mathbb{R}^d)^2$ and $T>0$. Then, one has:
    \begin{enumerate}
        \item If the following spectral condition is satisfied
         \begin{equation}\label{eq:SC forward synthesis}
         \sigma(DN(a_1))\cap\left\{i \frac{2 \pi \ell}{T} \in \mathbb{C} \mid\ell\in\Z\right\}=\emptyset
    \end{equation}
    then, the solution $a_f(\cdot)$ to~\eqref{eq::nonlinear control flow notation} corresponding to the input $I_f\in L^p(\R^d)$ given by
    \begin{equation}\label{eq:forward synthesis}
        I_f = \left[e^{TDN(a_1)}-\idty\right]^{-1}DN(a_1)(a_1-U_T(a_0))
    \end{equation}
    satisfies $\|a_f(T)-a_1\|_p =  \cO(T^2)$ as $T\to 0$.
     \item If the following spectral condition is satisfied
         \begin{equation}\label{eq:SC backward synthesis}
         \sigma(DN(a_0))\cap\left\{i \frac{2 \pi \ell}{T} \in \mathbb{C} \mid\ell\in\Z\right\}=\emptyset
    \end{equation}
    then, the solution $a_b(\cdot)$ to~\eqref{eq::nonlinear control flow notation} corresponding to the input $I_b\in L^p(\R^d)$ given by
    \begin{equation}\label{eq:backward synthesis}
        I_b = \left[\idty-e^{-TDN(a_0)}\right]^{-1}DN(a_0)(V_T(a_1)-a_0)
    \end{equation}
    satisfies $\|a_b(T)-a_1\|_p =  \cO(T^2)$ as $T\to 0$.
     \item If the following spectral condition is satisfied
         \begin{equation}\label{eq:SC backward nominal-state synthesis}
         \sigma(DN(V_T(a_1)))\cap\left\{i \frac{2 \pi \ell}{T} \in \mathbb{C} \mid\ell\in\Z\right\}=\emptyset
    \end{equation}
    then, the solution $a_{bn}(\cdot)$ to~\eqref{eq::nonlinear control flow notation} corresponding to the input $I_{bn}\in L^p(\R^d)$ given by
    \begin{equation}\label{eq:backward nominal-state synthesis}
        I_{bn} = \left[\idty - e^{-T DN(V_T(a_1))} \right]^{-1} DN(V_T(a_1))(V_T(a_1)-a_0)
    \end{equation}
    satisfies $\|a_{bn}(T)-a_1\|_p = \cO(T^2)$ as $T\to 0$.
    \end{enumerate}
\end{proposition}

\begin{remark}
The four synthesized inputs $I_{fn}$, $I_f$, $I_b$, and $I_{bn}$—defined respectively by equations~\eqref{eq:approximation of the forward nominal-state synthesis input}, \eqref{eq:forward synthesis}, \eqref{eq:backward synthesis}, and~\eqref{eq:backward nominal-state synthesis}—never coincide in general. This can be seen even in the trivial case where the initial state $a_0$ is an equilibrium of the drift dynamics $\dot{a}(t) = N(a(t))$, i.e., $N(a_0) = 0$.

In this situation, it is clear that $U_T(a_0) = a_0$ and $DN(U_T(a_0)) = DN(a_0)$ so that $I_b$ and $I_{bn}$ remain as defined in~\eqref{eq:backward synthesis} and~\eqref{eq:backward nominal-state synthesis}, while the forward nominal-state and final-state syntheses recast
\[
I_{fn} = \left[e^{T DN(a_0)} - \idty \right]^{-1} DN(a_0)(a_1 - a_0), 
\quad 
I_f = \left[e^{T DN(a_1)} - \idty \right]^{-1} DN(a_1)(a_1 - a_0).
\]
This clearly shows that the four control inputs are distinct. The only overlap occurs in the spectral conditions: $I_{fn}$ and $I_b$ are both well-defined under the same spectral condition~\eqref{eq:SC backward synthesis}, which is generally not the case when $N(a_0) \neq 0$ (see Proposition~\ref{pro:distinction of the spectrums} below).

Furthermore, if $a_1$ is also an equilibrium\footnote{This scenario is of practical interest---for example, when the goal is to drive the neural field from an unstable equilibrium $a_0$ to a stable equilibrium $a_1$ using constant in time input.} of the drift, then $V_T(a_1) = a_1$, $DN(V_T(a_1)) = DN(a_1)$, and
\[
I_b = \left[\idty - e^{-T DN(a_0)}\right]^{-1} DN(a_0)(a_1 - a_0),
\quad
I_{bn} = \left[\idty - e^{-T DN(a_1)}\right]^{-1} DN(a_1)(a_1 - a_0).
\]
Note, in particular, that $I_f$ and $I_{bn}$ are now well-defined under the same spectral condition~\eqref{eq:SC forward synthesis}. Altogether, even in this special case where both $a_0$ and $a_1$ are equilibria of the drift, we still have $I_{fn} \neq I_f \neq I_b \neq I_{bn}.$
\end{remark}

\subsection{On the spectral conditions and beyond}\label{ss:On the spectral conditions}

Here, we analyze the spectral conditions underlying the validity of each synthesis strategy. 
To begin this section, we state the following informative remark.

\begin{remark}\label{rmk:useful}
    It follows from Lemma~\ref{lem:on the SC} (let $t=T$ in~\eqref{eq:link between the SAs abstract}) that
    \begin{equation}\label{eq:useful}
          DN(U_T(a)) 
= DU_T(a) \, DN(a) \, DU_T(a)^{-1} 
+ \partial DU_T(a) \, N(a) \, DU_T(a)^{-1},\quad a\in L^p(\R^d).
    \end{equation}
    Then, $B_T:=DU_T(a) \, DN(a) \, DU_T(a)^{-1}$ is similar to $DN(a)$ so that $\sigma(B_T)=\sigma(DN(a))$. Next, the operator $\partial DU_T(a) \, N(a) \, DU_T(a)^{-1}$ vanishes identically when either $N(a) = 0$ or $T = 0$, since $\partial DU_T(a)\big|_{T=0} = 0$. Therefore, we deduce that for any $a \in L^p(\mathbb{R}^d)$ with $N(a) \neq 0$, one may have
    \[
    \sigma(DN(U_T(a))) \neq \sigma(DN(a)), \qquad T > 0.
    \]

    Nevertheless, the following result shows that these spectra remain close in the vicinity of $T\sim 0$, as established through perturbation theory for bounded linear operators~\cite{kato2013perturbation}. 
    
\end{remark}

\begin{proposition}\label{pro:distinction of the spectrums}
    Let $p \in [2, \infty]$, $a \in L^p(\mathbb{R}^d)$, and $T\ge 0$. Then, the following holds
    \begin{equation}\label{eq:perturbation of the spectrum}
        \sigma(DN(U_T(a)))\subseteq \sigma(DN(a))+\cO(T)\quad\text{as}\quad T\to 0.
    \end{equation}
    In particular, $\sigma(DN(U_T(a)))=\sigma(DN(a))$, if, either $N(a) = 0$ or $T = 0$.
\end{proposition}
\begin{proof}
Let $A_T:=DN(U_T(a))\in\mathscr{L}(L^p(\R^d))$. By~\eqref{eq:useful},
\[
A_T=B_T+E_T,\qquad 
B_T:=DU_T(a)\,DN(a)\,DU_T(a)^{-1},\qquad 
E_T:=\partial DU_T(a)\,N(a)\,DU_T(a)^{-1}.
\]
Since $B_T$ is similar to $DN(a)$, one has $\sigma(B_T)=\sigma(DN(a))$. Moreover, by
Lemma~\ref{lem:general estimates flows} (let $\gamma(t)=t$ and $\beta(t)=a$ in~\eqref{eq:spectral norm of D^2_phi} and evaluate the result at $t=T$),
\[
\|A_T-B_T\|=\|E_T\|=\mathcal{O}(T)\qquad\text{as }T\to0,
\]
that is, there exist $C>0$ and $T_0>0$ such that
$\|E_T\|\le CT$ for all $0<T<T_0$. 
Set\footnote{We denote by $\dist(z,{\rm S}):=\inf_{y\in{\rm S}}|z-y|$ the distance of $z$ from a subset ${\rm S}\subset\C$.} $\Gamma_{CT}:=\{z\in\C:\dist(z,\sigma(B_T))\ge CT\}\subset\rho(B_T)$ and define
\[
\delta_{CT}:=\min_{z\in\Gamma_{CT}}\|(B_T-z)^{-1}\|^{-1}>0.
\]
Since $\delta_{CT}>0$ (for each fixed $T$) and $\|E_T\|\to0$ as $T\to0$,
there exists $T_1\in(0,T_0)$ such that $\|E_T\|<\delta_{CT}$ for all $0<T<T_1$.
Therefore, \cite[Remark~3.2]{kato2013perturbation} implies that
$\Gamma_{CT}\subset\rho(A_T)$, and~\cite[Remark~3.3]{kato2013perturbation} yields
\[
\sigma(DN(U_T(a)))=\sigma(A_T)\subset\{z\in\C:\dist(z,\sigma(DN(a)))<CT\}.
\]
which is precisely \eqref{eq:perturbation of the spectrum}. The last part of the proposition follows from Remark~\ref{rmk:useful}.
\end{proof}

Given the relationships among these various spectral conditions, the remaining question is how to verify, in practice, whether a given condition is satisfied. We recall from Lemma~\ref{lem:on the SA in the linear case} and Remark~\ref{rmk:on the SA in the nonlinear settings} that these spectral conditions are automatically satisfied when
\[
\left(\mu < \mu_0 \quad \text{if} \quad 2 \leq p \leq \infty\right) \quad \text{or} \quad \left(\mu < \mu_1 \quad \text{if} \quad p = 2\right)
\]
where the threshold parameters $0 < \mu_0 \leq \mu_1$ are introduced in~\eqref{eq:mu criticals}. The main challenge lies in verifying the spectral conditions in the critical or supercritical regime, i.e., when $\mu \geq \mu_0$ or $\mu \geq \mu_1$ (in the case $p = 2$).

In the specific case of a constant $a\in L^\infty(\R^d)$, one has the following result. 

	\begin{proposition}\label{pro:on the spectrum of DN(a) in L^2}
	Let $\omega(-x) = \omega(x)$ for all $x\in\R^d$.
    If $\overline{a} \in L^\infty(\mathbb{R}^d)$ is constant (e.g., the homogeneous equilibrium of $\dot{a}(t) = N(a(t))$), then $DN(\overline{a}) \in \mathscr{L}(L^2(\mathbb{R}^d))$ is self-adjoint, and its spectrum is given by
		\begin{equation}\label{eq:on the spectrum of DN(a) in L^2 when a=0}
			\sigma(DN(\overline{a})) 
			= \left[-\alpha + \mu f'(\overline{a}) \min_{\xi \in \mathbb{R}^d} \widehat{\omega}(\xi), \,
			-\alpha + \mu f'(\overline{a}) \max_{\xi \in \mathbb{R}^d} \widehat{\omega}(\xi)\right] \subset \mathbb{R}.
		\end{equation}
\end{proposition}
\begin{proof}
For every $a\in L^\infty(\R^d)$, the operator $DN(a)\in\mathscr{L}(L^2(\R^d))$ is well-defined since $f'$ is bounded by Assumption~\ref{ass:general assumption}. If $a(x)=\overline{a}$ is constant, then
  \[
DN(\overline{a})b = -\alpha b+ \mu\,f'(\overline{a}) \omega \ast b,\qquad b\in L^2(\R^d).
\]
Since $\omega(-x)=\omega(x)$, $DN(\overline{a})$ is self-adjoint on $L^2(\R^d)$ so that its spectrum $\sigma(DN(\overline{a}))\subset\R$. Let $\lambda\in\R\bs\sigma(DN(\overline{a}))$, then $DN(\overline{a})-\lambda\idty$ is invertible, i.e., for all $\phi\in L^2(\R^d)$, there exists a unique $b\in L^2(\R^d)$ such that $DN(\overline{a})b-\lambda b=\phi$. Taking the Fourier transform yields
\[
 (m(\xi)-\lambda)\widehat{b}(\xi)=\widehat{\phi}(\xi),\qquad m(\xi):=-\alpha+\mu f'(\overline{a})\widehat{\omega}(\xi),\qquad\forall\xi\in\R^d.
\]
In particular, $m(\cdot)-\lambda$ is a real-valued, continuous, and bounded function on $\R^d$ by Riemann–Lebesgue lemma since $\omega\in L^1(\R^d)$. Therefore,
\[
\widehat{b}(\xi) = \widehat{\phi}(\xi)/(m(\xi)-\lambda) ,\qquad\xi\in\R^d
\]
is well-defined and belongs to $L^2(\R^d)$ if and only if $(m(\xi)-\lambda)\neq 0$, which is equivalent to
\[
\lambda\notin \left[-\alpha + \mu f'(\overline{a}) \min_{\xi \in \mathbb{R}^d} \widehat{\omega}(\xi), \,
                 -\alpha + \mu f'(\overline{a}) \max_{\xi \in \mathbb{R}^d} \widehat{\omega}(\xi)\right].
\]
This completes the proof of the proposition.
\end{proof}

	\begin{remark}\label{rmk:on the spectrum of DN(a) in L^2}
		The symmetry assumption $\omega(-x) = \omega(x)$ is standard in the literature, as most neural field models on unbounded domains $\mathbb{R}$ or $\mathbb{R}^2$ typically use kernels that are both homogeneous and isotropic, meaning $\omega(x,y) = \omega(|x - y|)$ \cite{bolelli2025neural,bressloff2001geometric, tamekue2024mathematical, tamekue2025reproducibility,coombes2005waves,faye2013localized, nicks2021understanding,laing2002multiple}. This is relevant to ensure that $ \sigma(DN(a)) \subset \mathbb{R} $ when $a\in L^\infty(\R
        ^d)$ is a real constant. 
			It follows from~\eqref{eq:on the spectrum of DN(a) in L^2 when a=0} that the spectral conditions~\eqref{eq:SC forward nominal-state synthesis},~\eqref{eq:SC forward synthesis},~\eqref{eq:SC backward synthesis}, and~\eqref{eq:SC backward nominal-state synthesis} are violated if and only if $ 0 \in \sigma(DN(\overline{a}))$ in the case where $a_0 = \overline{a}$, $a_1 = \overline{a}$, $U_T(a_0)=\overline{a}$ or $V_T(a_1)=\overline{a}$.  This occurs if and only if
				\[
				\mu = \frac{\alpha}{f'(\overline{a}) \lambda},\qquad \text{for some } \lambda \in \left[\min_{\xi \in \mathbb{R}^d} \widehat{\omega}(\xi),\, \max_{\xi \in \mathbb{R}^d} \widehat{\omega}(\xi)\right].
				\]
		\end{remark}

\section{Comments on the generic syntheses and perspectives}\label{ss:comments and open questions}

This section discusses the generic syntheses developed in Sections~\ref{ss:generic synthesis} and~\ref{s:Generic syntheses in practice}, outlining their complementarity, practical limitations, and possible directions for future investigation.

As already noted in Remark~\ref{rmk:on the range of p}, the restriction to the range $2 \le p \le \infty$ in the main results stated in Proposition~\ref{pro:expansion for forward nominal-state synthesis}, Theorems~\ref{thm:forward nominal-state synthesis} and~\ref{thm:forward nominal-state synthesis step function}, and Proposition~\ref{pro:the three other syntheses}, involves a subtle technical consideration. While the operator $DN(a) \in \mathscr{L}(L^p(\mathbb{R}^d))$ is indeed well-defined for all $1 < p \le 2$ by \cite[Lemma~B.8]{tamekue2024mathematical}, the operator $\varphi_{\tau,T}(I) \in \mathscr{L}(L^p(\mathbb{R}^d))$ involves the Gâteaux derivative $\partial DN$ of the Fréchet derivative $DN$ under the integral sign. However, Proposition~\ref{pro:second differential of N} only guarantees that $\partial DN(h)k \in \mathscr{L}(L^p(\mathbb{R}^d))$ for all $h,\,k \in L^p(\mathbb{R}^d)$ if $p\in[2, \infty]$. This raises an open question: is $DN$ Gâteaux differentiable on $L^p(\mathbb{R}^d)$ for $p\in(1,2)$? An affirmative answer would extend the validity of the proposed input syntheses to the full range $p\in(1, \infty]$. We refer to Remark~\ref{rmk:on L^1} for the specific case of $p=1$.

The proposed input syntheses~\eqref{eq:approximation of the forward nominal-state synthesis input}, \eqref{eq:forward synthesis}, \eqref{eq:backward synthesis}, and~\eqref{eq:backward nominal-state synthesis} apply to a broad class of nonlinear dynamical systems of the form \eqref{eq::nonlinear control flow notation}, namely
\[
\dot{a}(t) = N(a(t)) + I,
\]
where $a(t)$ evolves in a Banach space $\cX$, and the drift $N\in C^1(\cX)$, the Fréchet derivative $DN$ is Gâteaux differentiable, and the following uniform boundedness conditions are satisfied: there exists $C > 0$ such that
\[
\sup_{u\in\cX}\|DN(u)\| \le C, \qquad \sup_{u\in\cX}\|\partial DN(u)\| \le C.
\]

Finally, we emphasize that our proposed constant-in-time input syntheses do not directly extend to neural field models written in the \emph{activity-based} form~\cite{bressloff2011spatiotemporal,ermentrout1998neural,faugeras2008absolute,liley2001spatially}, namely
\begin{equation}\label{eq:Activity-based model}
    \partial_t a(x,t) = -\alpha\, a(x,t) + \mu\, f\!\big([\omega \ast a](x,t) + I(x)\big), 
    \quad a(x,0) = a_0(x), 
    \quad (x,t) \in \mathbb{R}^d \times \mathbb{R}_+.
\end{equation}
A detailed analysis of the extension of the proposed synthesis to~\eqref{eq:Activity-based model} is left for future work.

\section{How can the syntheses be used to predict visual illusions?}\label{s:application in predicting visual illusions}

Visual illusions are striking phenomena that shed light on the intricate workings of the visual system. It is well established that the visual system can produce illusory colors and shapes absent from the actual stimulus~\cite{grossberg1997visual}. Such illusions arise from specific patterns of neural activity---typically resulting from dynamic interactions within the visual cortex---and they often emerge almost instantaneously following exposure to a visual stimulus~\cite{billock2007neural, billock2012elementary,mackay1961visual, mackay1957moving}.


In modeling experimentally observed visual illusions using neural dynamics, the prevailing approach in the literature---see, e.g.,~\cite{bolelli2025neural,nicks2021understanding,tamekue2023cortical,tamekue2024mathematical,tamekue2025reproducibility}---relies on the stationary version of~\eqref{eq::nonlinear control flow notation}. In this framework, illusory perceptual patterns are interpreted as the inverse retinotopic image of a steady-state cortical activity profile driven by a visual stimulus via a corresponding sensory input.

This mechanistic approach offers a clear and interpretable link between the stimulus and the resulting perception. It is naturally suited to scenarios where the illusion emerges after the dynamics have stabilized, i.e., in the asymptotic regime, which is often appropriate given that many visual illusions can persist in the visual field for several seconds, see, for instance, \cite{mackay1957moving}. That said, several well-known illusions—including the visual MacKay effect (especially for the ``MacKay target'')~\cite{mackay1957moving}, the phenomena studied by Billock and Tsou~\cite{billock2007neural}, the Fraser spiral illusion~\cite{fraser1908new}, and the café wall illusion~\cite{kitaoka2004contrast}—can be perceived within very short (approximately $10$ sec. for the visual MacKay effect from the ``MacKay ray'' visual stimulus \cite{mackay1957moving}) latencies following stimulus onset. In such cases, cortical dynamics remain far from equilibrium, and the initial state of neural activity—modeled as $a_0$ in equation~\eqref{eq:NF-intro}—may significantly influence the illusory experience.

In this application, we should assume that the dimension of the space variable is $d=2$. As a prototypical example, suppose we aim to design a visual input $v_I$ that induces the perceptual pattern shown in \textbf{Fig.}~\ref{fig:tunnel-fovea}—denoted by $v_T$—after $T$ seconds of observation. This approach explicitly accounts for the initial cortical state, which can strongly influence the resulting perception. The pattern in \textbf{Fig.}~\ref{fig:tunnel-fovea} was originally reported in the psychophysical experiments of~\cite{billock2007neural}, where localized oriented stimuli in the visual field were shown to evoke compelling illusory perceptions in surrounding regions after brief exposures.

The control-theoretic formulation is as follows: Let $T > 0$ denote the small target time at which the illusory percept $v_T$ is observed. Let $a_1 := \mathscr{R}v_T$ represent the corresponding V1 cortical activity, obtained via the retinotopic map $\mathscr{R}$ (see, e.g.,~\cite{bressloff2001geometric,schwartz1977spatial})
    between the visual field (retina) and V1, and let $a_0$ denote the initial cortical state in V1. If we assume that the average activity in V1 evolves according to~\eqref{eq:NF-intro} (when neglecting the label that models V1's functional features, e.g., the sensitivity of V1's ``simple cells'' to local orientations from the stimuli in the retina~\cite{bressloff2001geometric}), the goal is to synthesize a constant-in-time input $I$ such that the solution $a(\cdot)$ to~\eqref{eq:NF-intro}  satisfies $\|a(\cdot, T) - a_1\|_p \le \varepsilon$
for a prescribed small $\varepsilon > 0$ in a suitable $L^p(\mathbb{R}^2)$ norm. The corresponding visual stimulus is then obtained via the inverse retinotopic map as $v_I = \mathscr{R}^{-1} I$.

\begin{figure}[t!]
    \begin{subfigure}[b]{0.49\textwidth}
        \centering
        \includegraphics[width=0.5\textwidth]{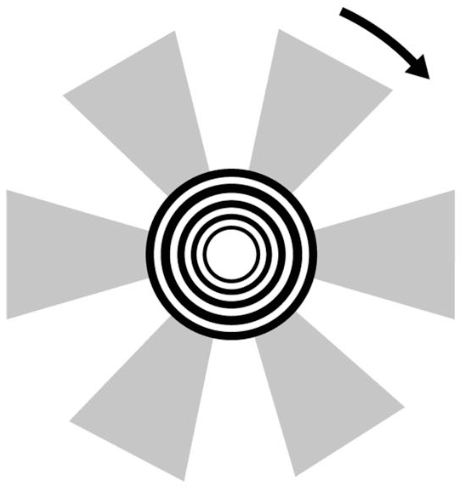}
          \caption{\footnotesize Geometrical pattern $v_T$ to induce in the visual field. The arrow indicates the direction of apparent motion \cite{billock2007neural}.}
          \label{fig:tunnel-fovea}
    \end{subfigure}
    \hfill
    \begin{subfigure}[b]{0.49\textwidth}
        \centering
  \includegraphics[width=0.5\textwidth]{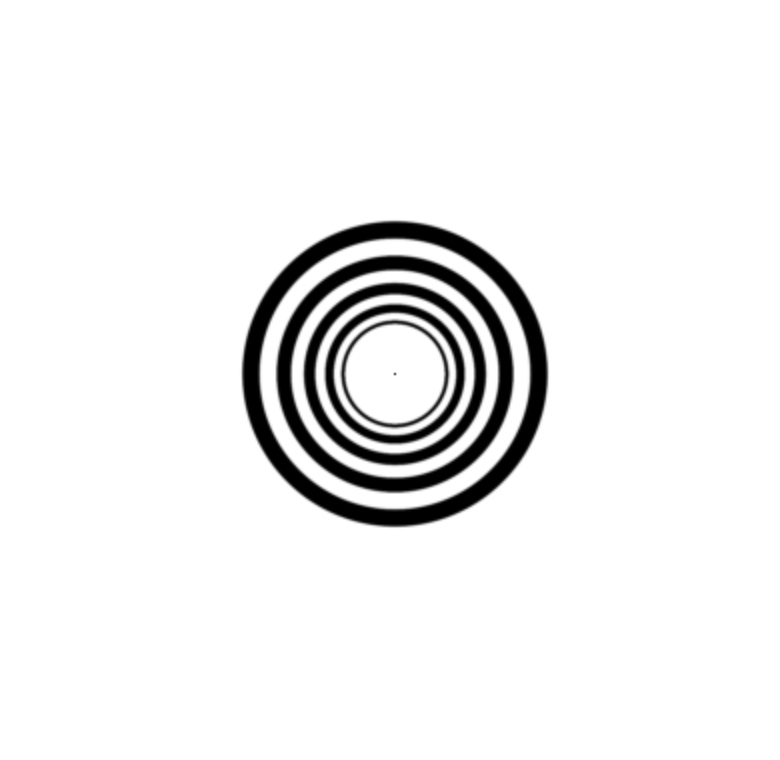}
    \caption{\footnotesize Visual stimulus $v_I$ that induces the illusion $v_T$ in Fig.~\ref{fig:tunnel-fovea}, as reported in~\cite{billock2007neural}.}
    \label{fig:tunnel-fovea-stimul}
    \end{subfigure}
    \vspace{-0.5cm}
    \end{figure}
    
Since $a_1$ is fixed and $I$ can be synthesized using any of the formulations~\eqref{eq:approximation of the forward nominal-state synthesis input}, \eqref{eq:forward synthesis}, \eqref{eq:backward synthesis}, or~\eqref{eq:backward nominal-state synthesis}, the primary degree of freedom—aside from the intrinsic parameters $\omega$, $f$, $\mu$, and $\alpha$ in~\eqref{eq:NF-intro}, which also influence the synthesis—is the choice of the initial state $a_0$. For example, in the case of the target pattern $v_T$ in \textbf{Fig.}~\ref{fig:tunnel-fovea},~\cite{billock2007neural} provides a stimulus $v_I$---depicted in \textbf{Fig.}~\ref{fig:tunnel-fovea-stimul}---known to induce this illusion. Therefore, the input $I$ is said to predict a visual illusion if its inverse retinotopic map coincides with an experimentally validated visual stimulus.

We defer to future work the application of the proposed syntheses for predicting visual illusions, as outlined in this section. This endeavor will require collaboration with researchers in visual perception science to experimentally validate the synthesized inputs, following approaches similar to those in~\cite{billock2007neural}.

\section{Numerical illustration of the proposed constant-in-time input syntheses}\label{s:NI}

This section presents numerical illustrations of the control syntheses developed in Section~\ref{ss:generic synthesis}, implemented in the scientific computing language \texttt{Julia}~\cite{bezanson2017julia,kelley2022solving}. 
Our goal is to demonstrate their practical effectiveness through simulations on one- and two-dimensional spatial domains. Specifically, we consider the approximate constant-in-time inputs $I_{fn}(\cdot)$, $I_f(\cdot)$, $I_b(\cdot)$, and $I_{bn}(\cdot)$---defined for a short time horizon $T = \tau > 0$ by equations~\eqref{eq:approximation of the forward nominal-state synthesis input}, \eqref{eq:forward synthesis}, \eqref{eq:backward synthesis}, and~\eqref{eq:backward nominal-state synthesis}, respectively---and solve~\eqref{eq:NF-intro} under each input to obtain the corresponding final states $a_{fn}(\cdot, T)$, $a_f(\cdot, T)$, $a_b(\cdot, T)$, and $a_{bn}(\cdot, T)$.
We then evaluate the endpoint errors
\begin{equation}\label{eq:endpoint errors}
    \|a_{fn}(\cdot, T) - a_1(\cdot)\|_p,\quad 
    \|a_f(\cdot, T) - a_1(\cdot)\|_p,\quad 
    \|a_b(\cdot, T) - a_1(\cdot)\|_p,\quad 
    \text{and} \quad 
    \|a_{bn}(\cdot, T) - a_1(\cdot)\|_p,
\end{equation}
in the $L^p(\mathbb{R}^d)$ norm to verify that they remain small, consistent with the theoretical guarantees.

To assess the accuracy and efficiency of the proposed generic syntheses, we also compare (see~\textbf{Fig.}~\ref{fig:all_strategies_error_input_norms}) them to standard input syntheses based on the linearization of the nonlinear equation~\eqref{eq::nonlinear control flow notation} about the initial and target states, $a_0$ and $a_1$, respectively. The corresponding linearized inputs are given by
\begin{equation}\label{eq:linearized input syntheses}
    \begin{split}
        I_{a_0} &= -N(a_0) + \left[e^{T DN(a_0)} - \idty\right]^{-1} DN(a_0)\big(a_1 - e^{T DN(a_0)} a_0\big), \\
        I_{a_1} &= -N(a_1) + \left[e^{T DN(a_1)} - \idty\right]^{-1} DN(a_1)\big(a_1 - e^{T DN(a_1)} a_0\big),
    \end{split}
\end{equation}
and the corresponding solutions to~\eqref{eq::nonlinear control flow notation} are denoted $a_{a_0}(\cdot)$ and $a_{a_1}(\cdot)$. Observe that $I_{a_0}=I_{fn}$ and $I_{a_1}=I_{bn}$ if both $a_0$ and $a_1$ are equilibria of the drift dynamics $\dot{a}(t)=N(a(t))$.

All simulations are performed in Julia using the \texttt{DifferentialEquations} package, where~\eqref{eq:NF-intro} is integrated via the \texttt{solve} routine. The convolution kernel is constructed using \texttt{centered} and \texttt{imfilter} functions from the \texttt{ImageFiltering} package. To compute the Jacobian-vector product $DN(a)b$, rather than relying on finite-difference approximations as commonly suggested in~\cite{kelley2022solving}---which may introduce numerical inaccuracies---we employ automatic differentiation via the \texttt{ForwardDiff} package. To evaluate $e^{T DN(a)}$, we construct the Jacobian matrix $DN(a)$ by applying it to each canonical basis vector $e_i$ of $\mathbb{R}^M$, where $M$ is the number of spatial discretization points.


Finally, it is worth noting that the parameter $\mu>0$ in each example is expressly chosen in such a way that Theorem~\ref{thm:Banach fixed-point theorem} cannot be applied.

\subsection{Examples in a one-dimensional spatial domain}\label{ss:NS in 1D}
As noted in the introduction, this example is biologically motivated: it demonstrates that a bump-like activity pattern---commonly associated with neural encoding in working or short-term memory~\cite{laing2002multiple}---can be successfully induced over a short time horizon $T = 0.25$, even in a subcritical regime $\mu < \mu_c$, where such patterns do not naturally arise.

To this end, we consider the Amari-type neural field equation~\eqref{eq:NF-intro} posed on a one-dimensional domain, with parameters drawn from~\cite{faye2013localized}:
\begin{equation}\label{eq:wizard 1D}
    \omega(x) = 2\widehat{\omega}_c(\widehat{\omega}_c+1)e^{-|x|}-\frac{(2\widehat{\omega}_c+1)^2}{2}\sqrt{\frac{\widehat{\omega}_c}{\widehat{\omega}_c+1}}e^{-\sqrt{\frac{\widehat{\omega}_c}{\widehat{\omega}_c+1}}|x|},\quad
    f(x) = \frac{1}{1+\exp(-x+\theta)} - \frac{1}{1+\exp(\theta)},
\end{equation}
where $\widehat{\omega}_c := \max_{\xi \in \mathbb{R}} \widehat{\omega}(\xi) = \widehat{\omega}(\pm1/2\pi) = 5$ and $\theta = 3.5$. Note also from~\cite{faye2013localized} that $\widehat{\omega}(0) = \min_{\xi \in \mathbb{R}} \widehat{\omega}(\xi) = -1$.

\begin{figure}[t!]
    \centering
    \includegraphics[width=0.9\linewidth]{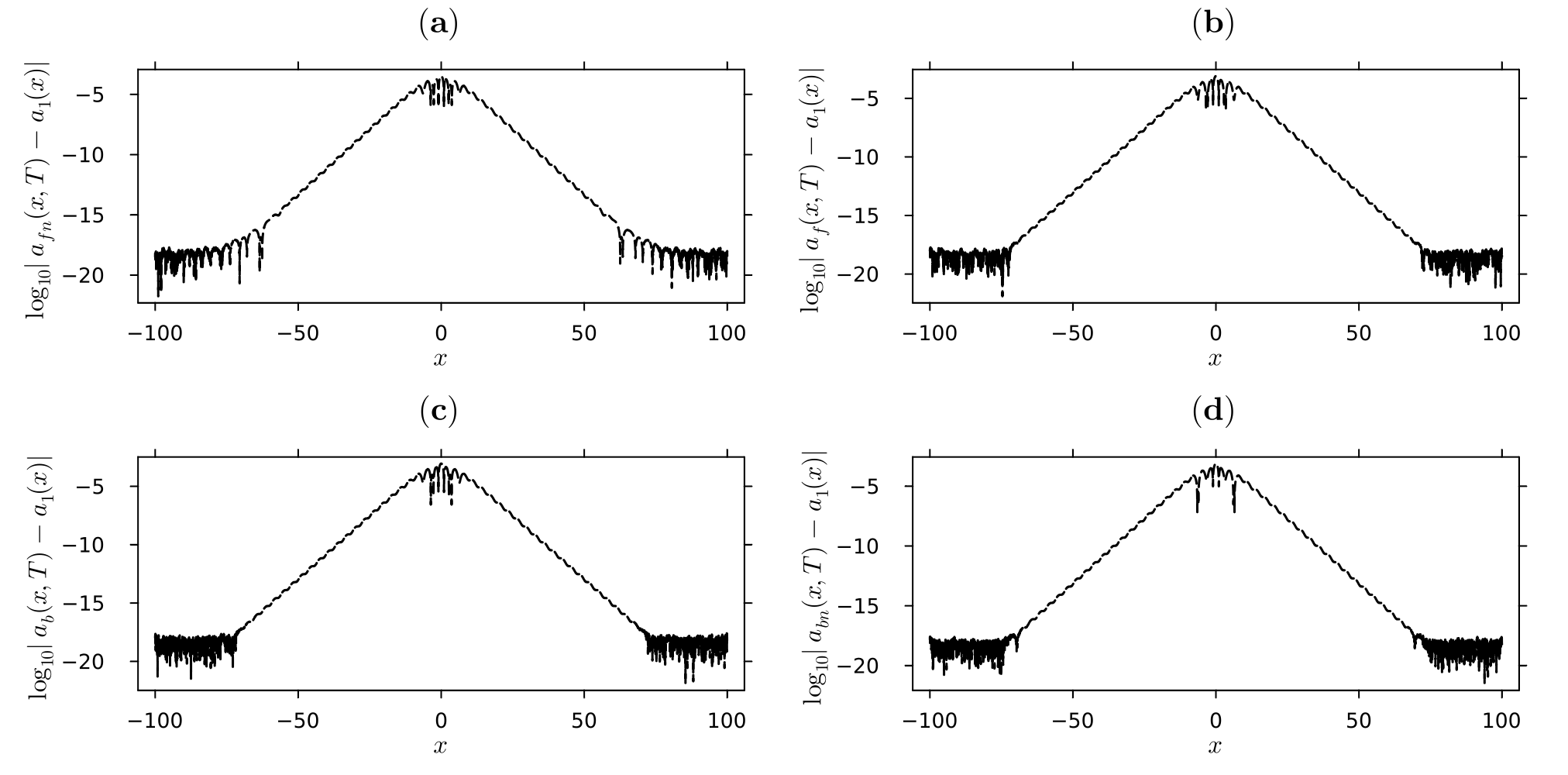}
    \caption{
        Log-scale plot of the endpoint mismatch $\log_{10} \lvert a(x, T) - a_1(x) \rvert$ for each of the four control strategies, with $a_0$ and $a_1$ defined in~\eqref{eq:initial and target state faye} and final time $T = 0.25$. 
        \textbf{(a)} Forward nominal-state input: $\|a_{fn}(\cdot, T) - a_1(\cdot)\|_2 = 4.4 \times 10^{-4}$, $\|a_{fn}(\cdot, T) - a_1(\cdot)\|_\infty = 3.3 \times 10^{-4}$. 
        \textbf{(b)} Forward final-state input: $\|a_f(\cdot, T) - a_1(\cdot)\|_2 = 9.9 \times 10^{-4}$, $\|a_f(\cdot, T) - a_1(\cdot)\|_\infty = 7.9 \times 10^{-4}$. 
        \textbf{(c)} Backward initial-state input: $\|a_b(\cdot, T) - a_1(\cdot)\|_2 = 1.1 \times 10^{-3}$, $\|a_b(\cdot, T) - a_1(\cdot)\|_\infty = 8.8 \times 10^{-4}$. 
        \textbf{(d)} Backward nominal-state input: $\|a_{bn}(\cdot, T) - a_1(\cdot)\|_2 = 9.4 \times 10^{-4}$, $\|a_{bn}(\cdot, T) - a_1(\cdot)\|_\infty = 7.7 \times 10^{-4}$. 
        The mismatch is sharply localized around $x = 0$, where the target bump is centered and where the derivatives $f'\big(U_T(a_0)\big)$, $f'(a_1)$, $f'(a_0)$, and $f'\big(V_T(a_1)\big)$ reach their maximum values, amplifying sensitivity to perturbations. Here $\omega$ and $f$ are given in \eqref{eq:wizard 1D}.\vspace{-0.5cm}
    }
    \label{fig:diff_final-target_states}
\end{figure}

According to~\cite{faye2013localized}, in the absence of external input, the trivial state $a \equiv 0$ is the unique exponentially stable equilibrium for $\mu < \mu_c$, and a bifurcation occurs at $\mu = \mu_c$. At this point, two symmetric branches of spatially localized solutions emerge, which are approximately given by
\[
a(x) = 2\sqrt{\frac{-2c_1^1\lambda}{c_3^0}}\,\operatorname{sech}\left(x\sqrt{c_1^1\lambda}\right)\cos(x + \phi) + \mathcal{O}(\lambda), \quad x \in \mathbb{R},
\]
for small $\lambda < 0$, where $\phi \in \{0, \pi\}$, and $c_1^1, c_3^0 < 0$ are constants depending on $f$, $\omega$, and $\mu_c$.

For the simulations, we fix $\phi = 0$, neglect higher-order terms in $\lambda$, and assume $c_1^1\lambda = 0.0625$ and $c_3^0 = -2$. The initial and target states are defined as
\begin{equation}\label{eq:initial and target state faye}
    a_0(x) = 0, \quad a_1(x) = 0.5\,\operatorname{sech}(0.25x)\cos(x), \quad x \in \mathbb{R},
\end{equation}
both belonging to $L^2(\mathbb{R}) \cap L^\infty(\mathbb{R})$. As in~\cite{faye2013localized}, we set $\alpha = 1$. Then, by Remark~\ref{rmk:on the spectrum of DN(a) in L^2} and the definition of $\mu_c$ in~\eqref{eq:mu criticals}, the inputs $I_{fn}$ and $I_b$ are well defined whenever
\[
\mu \in (0, \mu_c) = (0, 7.02).
\]
Likewise, by Remark~\ref{rmk:on the SA in the nonlinear settings} and the definition of $\mu_1$ in~\eqref{eq:mu criticals}, the inputs $I_f$ and $I_{bn}$ are well defined if
\[
\mu \in (0, \mu_1) = (0, 0.8).
\]
To ensure uniform applicability across all synthesis strategies, we set $\mu = 0.625\,\mu_1 = 0.5$ in the experiments.

For implementation, the spatial domain $\Omega = [-L, L]$ with $L = 100$ is discretized with mesh width $\Delta x = 0.025$, yielding $M = 1 + \frac{2L}{\Delta x} = 8001$ uniformly spaced grid points. To assess the effectiveness of each control, we compute the endpoint errors~\eqref{eq:endpoint errors} over $\Omega$ in both the $L^2$ and $L^\infty$ norms. See \textbf{Fig.}~\ref{fig:diff_final-target_states} for detailed results.

\begin{figure}[t!]
    \centering
    \includegraphics[width=0.9\linewidth]{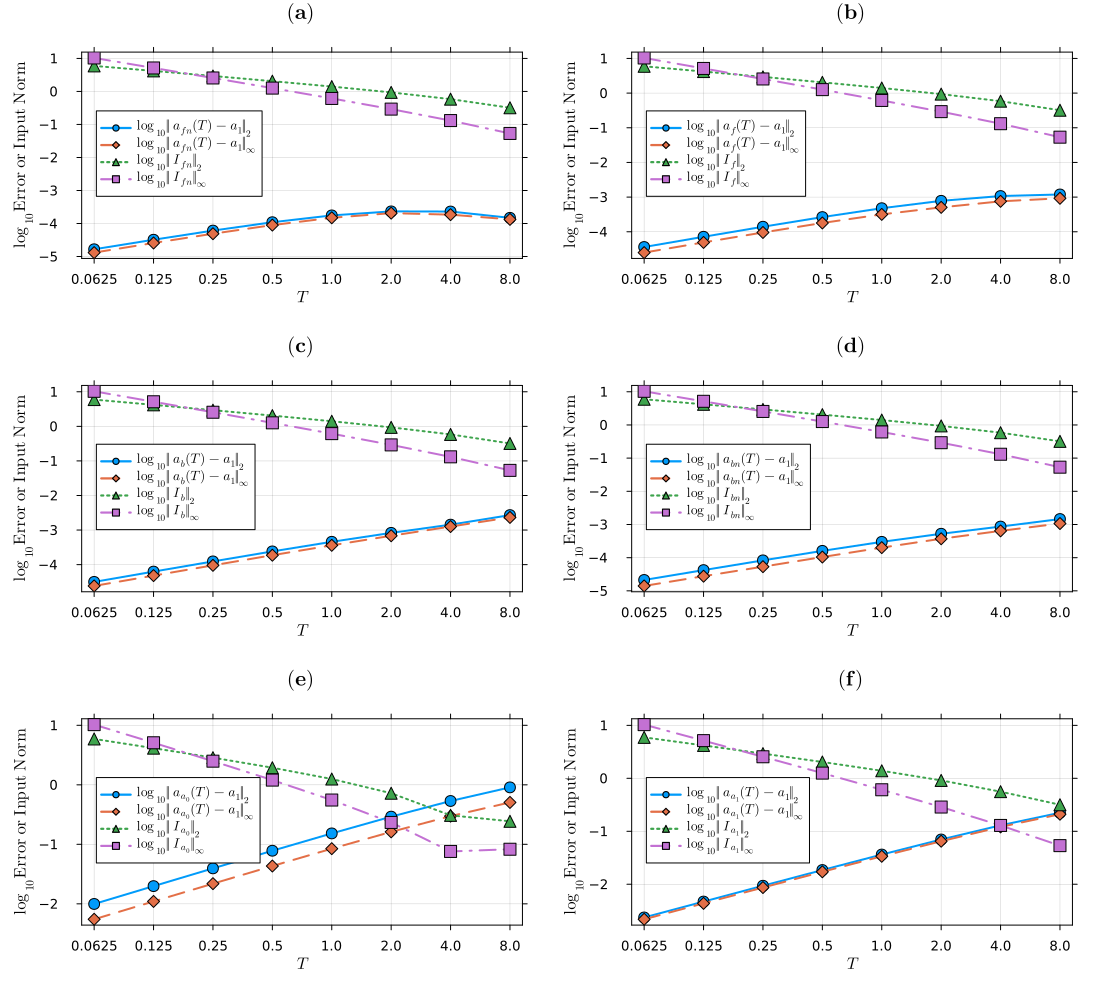}
    \caption{
    Comparison of $\log_{10}$ endpoint errors and input norms for each control synthesis strategy as a function of the time horizon $T \in \{0.0625,\,0.125,\,0.25,\,0.5,\,1,\,2,\,4,\,8\}$, with $\omega$ and $f$ defined in~\eqref{eq:DoG 1D} and $a_0$, $a_1$ given in~\eqref{eq:considered initial and target state}, satisfying $\|a_1 - a_0\|_{L^2(\Omega)} = 1.49$ and $\|a_1 - a_0\|_{L^\infty(\Omega)} = 0.65$. 
    Each panel shows the $\log_{10}$ of the $L^2(\Omega)$ and $L^\infty(\Omega)$ endpoint errors and input norms for:
    \textbf{(a)} forward nominal-state input $I_{fn}$,
    \textbf{(b)} forward final-state input $I_f$,
    \textbf{(c)} backward initial-state input $I_b$,
    \textbf{(d)} backward nominal-state input $I_{bn}$,
    \textbf{(e)} input linearized at the initial state $I_{a_0}$,
    \textbf{(f)} input linearized at the target state $I_{a_1}$. 
    As $T$ increases, the input norms generally decrease for all syntheses except $I_{a_0}$, which decreases up to $T=4$ and then rises. 
    For small time horizons, the forward and backward nominal-state syntheses yield the lowest endpoint errors, followed by the backward initial-state and forward final-state syntheses. 
    Notably, both linearized inputs $I_{a_0}$ and $I_{a_1}$ given in~\eqref{eq:linearized input syntheses} perform significantly worse than the generic syntheses in terms of both error and input norm. 
    }
    \label{fig:all_strategies_error_input_norms}
\end{figure}

In the second example, we consider the neural field equation~\eqref{eq:NF-intro} on $\mathbb{R}$ with a Difference-of-Gaussians (DoG) connectivity kernel and a shifted sigmoid activation function
\begin{equation}\label{eq:DoG 1D}
    \omega(x) = \frac{1}{\sigma_1\sqrt{2\pi}}e^{-\frac{x^2}{2\sigma_1^2}} - \frac{\kappa}{\sigma_2\sqrt{2\pi}}e^{-\frac{x^2}{2\sigma_2^2}}, \qquad 
    f(x) = \frac{1}{1 + e^{-x + 0.25}} - \frac{1}{1 + e^{0.25}}, \quad x \in \mathbb{R},
\end{equation}
with parameters $\sigma_1 = 1/(\pi\sqrt{2})$, $\sigma_2 = \sqrt{2}/\pi$, and $\kappa = 0.95$. We set $\alpha = 0.1$ and $\mu = 0.75 \mu_1 = 0.624$, where $\mu_1 = 0.832 < \mu_c = 0.845$ is computed using~\eqref{eq:mu criticals}. The initial and target states are defined by
\begin{equation}\label{eq:considered initial and target state}
     a_0(x) = \frac{1}{\sqrt{1 + x^2}},\qquad  a_1(x) = e^{-x^2} - 0.5 e^{-0.5 x^2},\qquad x\in\mathbb{R},
\end{equation}
 \begin{figure}[t!]
    \centering
    \includegraphics[width=0.9\linewidth]{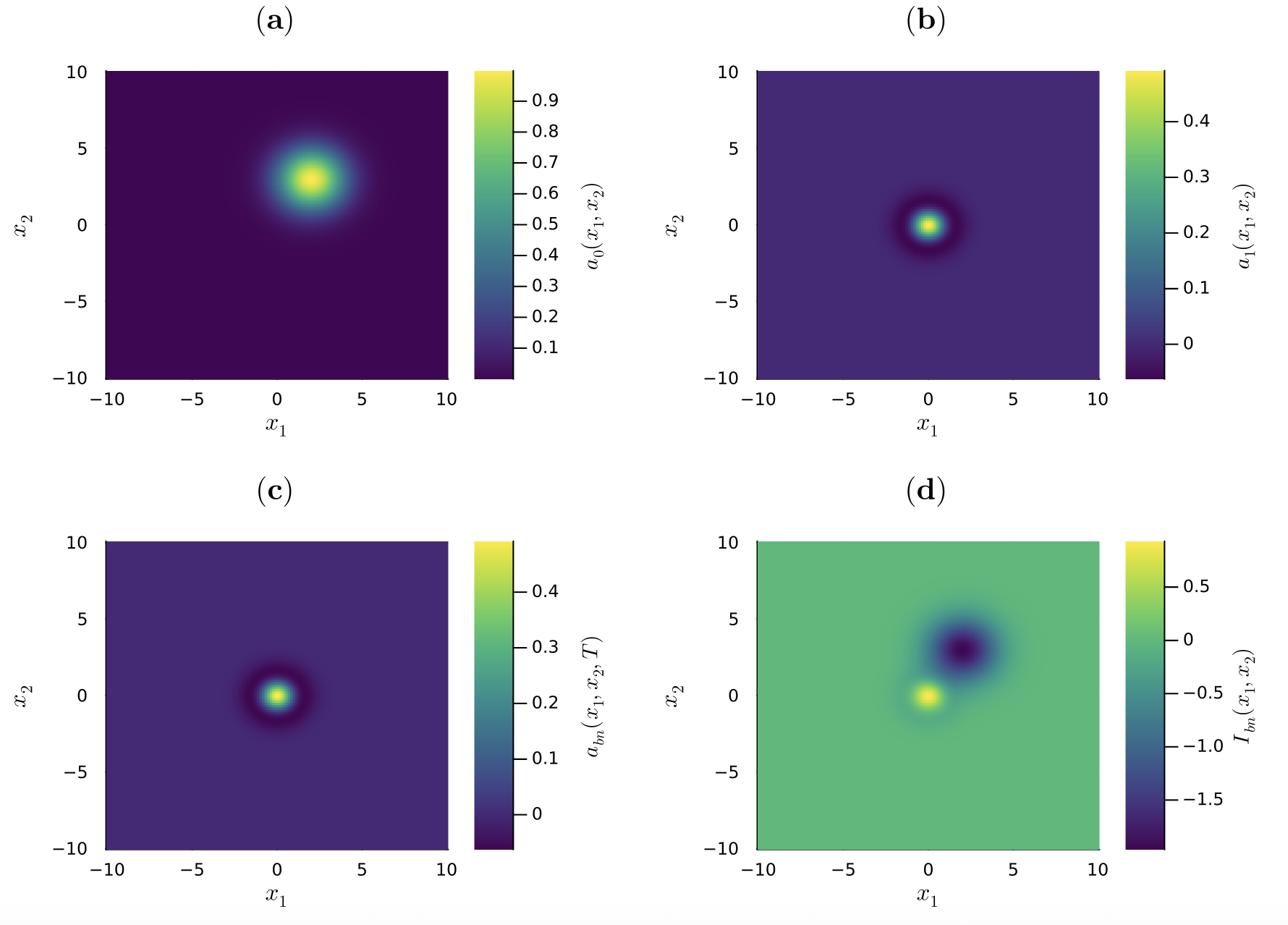}
    \caption{
        Numerical simulation of the two-dimensional neural field model~\eqref{eq:NF-intro} controlled via the backward nominal-state input $I_{bn}$ over the time horizon $T = 0.5$, using parameters from~\eqref{eq:DoG 2D} and initial/target states from~\eqref{eq:initial and target state 2D}. 
        \textbf{(a)} Initial state $a_0(x_1,x_2)$. 
        \textbf{(b)} Target state $a_1(x_1,x_2)$. 
        \textbf{(c)} Final state $a_{bn}(x_1,x_2,T)$. 
        \textbf{(d)} Input $I_{bn}(x_1,x_2)$ used throughout $[0,T]$. 
        The endpoint errors are $\|a_{bn}(T) - a_1\|_2 = 1.7 \times 10^{-4}$ and $\|a_{bn}(T) - a_1\|_\infty = 8.4 \times 10^{-5}$, while the input norms are $\|I_{bn}\|_2 = 5.05$ and $\|I_{bn}\|_\infty = 1.96$. For reference, the initial discrepancy is $\|a_1 - a_0\|_2 = 2.57$ and $\|a_1 - a_0\|_\infty = 0.99$.
    }
    \label{fig:in 2D}
\end{figure}
both of which belong to $L^2(\mathbb{R}) \cap L^\infty(\mathbb{R})$. According to Amari's terminology~\cite{amari1977dynamics}, the initial state $a_0$ can be interpreted as an ``$\infty$-solution'' corresponding to a spatially widespread excitation across the entire domain, while the target state $a_1$ represents a ``1-bump'' solution, characterized by a localized region of activity. 

Under this setting, all constant-in-time inputs $I_{fn}$, $I_f$, $I_b$, $I_{bn}$, $I_{a_0}$, and $I_{a_0}$ are well-defined. For simulations, the spatial domain $\Omega = [-L, L]$ with $L = 20$ is discretized using a uniform mesh with step size $\Delta x = 0.01$, yielding $M = 1 + \frac{2L}{\Delta x} = 4001$ grid points. To evaluate the performance of each control strategy, we compute the endpoint errors $\log_{10}\|a(\cdot, T) - a_1\|_2$ and $\log_{10}\|a(\cdot, T) - a_1\|_\infty$, along with the norms of the corresponding inputs across various time horizons $T \in \{0.0625,\,0.125,\,0.25,\,0.5,\,1,\,2,\,4,\,8\}$. We refer to Fig.~\ref{fig:all_strategies_error_input_norms} for visualization.

\subsection{Example in two-dimensional spatial domain}\label{ss:NS in 2D}
We now consider~\eqref{eq:NF-intro} posed on $\mathbb{R}^2$ with a difference-of-Gaussians (DoG) lateral connectivity kernel and a shifted sigmoid activation function
\begin{equation}\label{eq:DoG 2D}
    \omega(x) = \frac{1}{2\pi\sigma_1^2}e^{-\frac{|x|^2}{2\sigma_1^2}} - \frac{\kappa}{2\pi\sigma_2^2}e^{-\frac{|x|^2}{2\sigma_2^2}}, \qquad 
    f(s) = \frac{1}{1 + e^{-s + 0.25}} - \frac{1}{1 + e^{0.25}}, \quad (x, s) \in \R^2\times\R,
\end{equation}
with parameters $\sigma_1 = 1/(\pi\sqrt{2})$, $\sigma_2 = \sqrt{2}/\pi$, and $\kappa = 0.85$. The objective is to steer the system from the initial state $a_0$ to the target state $a_1$ in time $T = 0.5$, where
\begin{equation}\label{eq:initial and target state 2D}
     a_0(x) = e^{-\left(\frac{(x_1-2)^2}{4} + \frac{(x_2-3)^2}{4}\right)},\qquad  
     a_1(x) = e^{-|x|^2} - 0.5\,e^{-\frac{|x|^2}{2}},\qquad x = (x_1,x_2) \in \R^2,
\end{equation}
both belong to $L^p(\mathbb{R}^2)$. We fix $\alpha=0.1$ and choose $\mu := 0.75 \mu_1 = 0.6$, where $\mu_1 = 0.8$ is obtained using~\eqref{eq:mu criticals}. 

The domain $\mathbb{R}^2$ is approximated by the square $\Omega = [-L, L]^2$ with $L = 10$, discretized on a uniform grid with mesh size $\Delta x = \Delta y = 0.15$, yielding $M = \left(1 + \frac{2L}{\Delta x}\right)^2 = 17956$ grid points. In this example, we focus solely on the backward nominal-state input, $I_{bn}$. See \textbf{Fig.}~\ref{fig:in 2D} for visualization. 

\section{Conclusion}
In this article, we have investigated the controllability properties of Amari-type neural field equations, with an emphasis on synthesizing piecewise constant-in-time and constant-in-time inputs to steer neural activity from an initial to a desired target state. By constructing explicit constant-in-time input formulas, we extend the theoretical understanding of how sustained inputs can be designed to modulate population activity, opening the door to potential applications in neuroscience, psychophysics, and neurostimulation.  

Our numerical illustrations demonstrate the technical feasibility and effectiveness of the proposed input syntheses~\eqref{eq:approximation of the forward nominal-state synthesis input}, \eqref{eq:forward synthesis}, \eqref{eq:backward synthesis}, and~\eqref{eq:backward nominal-state synthesis}. In particular, the simulations confirm that these generic syntheses yield significantly smaller endpoint errors and more efficient control efforts compared to naive synthesis (see~\eqref{eq:linearized input syntheses}), based on linearizing the dynamics at the initial or target state when it is not an equilibrium of the drift dynamics. 
Moreover, from a computational standpoint, these generic syntheses are obtained by closed-form algebraic formulas once the flow maps $U_T(a_0)$ or $V_T(a_1)$ are evaluated; numerically, this requires only a forward or backward integration of the drift dynamics, together with standard linear-algebra operations (matrix exponential and a linear solve), and avoids iterative optimization routines.

A natural direction for future work is to integrate experimental data to identify physiologically plausible models and experimentally validate the proposed synthesis framework---for example, in the context of predicting and mechanistically explaining paradoxical visual illusions as outlined in Section~\ref{s:application in predicting visual illusions}.


\appendix

\section{Proofs of the main results}\label{s:proofs of main results}

In this section, we present the proof of the main results provided in the main text.

\subsection{Proof of Theorem~\ref{thm:Banach fixed-point theorem}}\label{ss:proof of the Banach fixed-point theorem}

\begin{proof}
First, we consider the nonlinear map $\Phi:L^p(\R^d)\to L^p(\R^d)$ defined by
\begin{equation}
    \Phi(I) = \alpha(1 - e^{-\alpha T})^{-1} \left\{ a_1 - e^{-\alpha T} a_0 - \mu \int_0^T e^{-\alpha(T - t)} \, \omega \ast f(a_I(t)) \, dt \right\},\quad I\in L^p(\R^d)
\end{equation}
where $a_I(\cdot)\in C^0([0, T]; L^p(\R^d)) $ is the unique solution to~\eqref{eq:NF-intro} with the input $I\in L^p(\R^d)$, provided by Proposition~\ref{pro:well-posedness}. To complete the proof of the theorem, it suffices to show that $\Phi$ is $L(T)-$Lipschitz continuous, where $L(T)>0$ is defined in \eqref{eq:lipschitz constant of Phi}. Let $J\in L^p(\R^d)$, and $a_J(\cdot)$ be the corresponding solution to~\eqref{eq:NF-intro}. One has 
\[
\Phi(J)-\Phi(I) = \alpha(1 - e^{-\alpha T})^{-1}\mu\int_0^T e^{-\alpha(T - t)} \omega\ast[f(a_I(t))-f(a_J(t))]\,dt
\]
which taking the $L^p(\R^d)$-norm, and Young-convolution inequality yield
\begin{equation}\label{eq:term 1}
    \|\Phi(J)-\Phi(I)\|_p\le \alpha(1 - e^{-\alpha T})^{-1}\mu\|\omega\|_1\int_0^Te^{-\alpha(T - t)}\|a_I(t)-a_J(t)\|_p\,dt.
\end{equation}
Using \eqref{eq:implicit sol representation}, and Gronwall's lemma, one obtains
\begin{equation}\label{eq:term 2}
    \|a_I(t)-a_J(t)\|_p\le e^{-\alpha t\left(1-\frac{\mu}{\mu_0}\right)}\|I-J\|_p\int_0^te^{\alpha s\left(1-\frac{\mu}{\mu_0}\right)}\, ds.
\end{equation}
Combining \eqref{eq:term 1} and \eqref{eq:term 2}, one deduces
\begin{equation}\label{eq:term 3}
    \|\Phi(J)-\Phi(I)\|_p\le\alpha^2\frac{\mu}{\mu_0}\|I-J\|_p(1 - e^{-\alpha T})^{-1}\int_0^Te^{-\alpha(T-t)}e^{-\alpha t\left(1-\frac{\mu}{\mu_0}\right)}\int_0^te^{\alpha s\left(1-\frac{\mu}{\mu_0}\right)}\,ds\,dt.
\end{equation}
Now, if $\mu=\mu_0$, one finds from \eqref{eq:term 3} that 
\begin{equation}\label{eq:term 4}
     \|\Phi(J)-\Phi(I)\|_p\le\alpha^2\frac{\mu}{\mu_0}\|I-J\|_p(1 - e^{-\alpha T})^{-1}\int_0^Te^{-\alpha(T-t)}t\,dt=\left[\frac{\alpha T}{1-e^{-\alpha T}}-1\right]\|I-J\|_p
\end{equation}
by using integration by parts of $\alpha\int_0^Te^{-\alpha(T-t)}t\,dt$. In the case of $\mu\neq\mu_0$, one deduces from \eqref{eq:term 3} that
\begin{equation}\label{eq:term 5}
    \|\Phi(J)-\Phi(I)\|_p\le \frac{1}{r-1}\left(\frac{e^{\alpha rT}-1}{e^{\alpha T}-1}-r\right)\|I-J\|_p,\qquad r:=\frac{\mu}{\mu_0}
\end{equation}
by performing the integration of 
\[
\int_0^Te^{-\alpha(T-t)}e^{-\alpha t\left(1-\frac{\mu}{\mu_0}\right)}\int_0^te^{\alpha s\left(1-\frac{\mu}{\mu_0}\right)}\,ds\,dt=\frac{1}{\alpha^2(1-r)}\left[\left(1-e^{-\alpha T}\right)+\frac{\left(1-e^{\alpha r T}\right)}{re^{\alpha T}}\right].
\]
Inequalities~\eqref{eq:term 4} and \eqref{eq:term 5} show that $\Phi$ is $L(T)-$Lipschitz continuous. Finally, an immediate real analysis shows that $L(T)<1$ if and only if the conditions in the statement of the theorem are satisfied.
\end{proof}



\subsection{Proof of Theorem~\ref{thm:backward representation}}\label{ss:proof of backward representation}
\begin{proof}
     Let $t\in [0, T]$ and set $b:s\in[0, t]\mapsto b(s)\in L^p(\R^d)$ so that $a(s) = V_{T-s}(b(s))$ is the solution of \eqref{eq::nonlinear control flow notation}. Then $b$ is derivable almost everywhere w.r.t. $s$ and it holds
    \begin{equation}
        N(a(s))+I(s)=\dot{a}(s)=N(a(s))+DV_{T-s}(b(s))\dot{b}(s)
    \end{equation}
    so that using \eqref{inverse of the differential of the flow}, we find that $b$ solves the following
    \begin{equation}\label{eq:equation satisfying b}
        \dot{b}(s)=DU_{T-s}(a(s))I(s),\qquad b(0)=U_T(a_0).
    \end{equation}
    Integrating \eqref{eq:equation satisfying b} over $[0, t]$ yields \eqref{eq:backward representation}. Conversely, if \eqref{eq:backward representation} holds, then $a(0)=a_0$ and $a\in C^0([0, T]; L^p(\R^d))$. 
    Otherwise, there exists (a nondecreasing sequence) $(t_n)\subset [0, T]$, $t_{*}\in[0, T]$ with $t_n\to t_{*}$ and $\varepsilon>0$ such that $\|U_{T-t_n}(a(t_n))-U_{T-t_{*}}(a(t_{*}))\|_p\ge\varepsilon$. In fact, $a\in C^0([0, T]; L^p(\R^d))$ if and only if $b(t):=U_{T-t}(a(t))\in C^0([0, T]; L^p(\R^d))$ since $U_{T-t}$ is invertible and $C^1$ w.r.t. $t\in\R$. However, we have from \eqref{eq:backward representation} that
    \begin{equation}
       \|U_{T-t_n}(a(t_n))-U_{T-t_{*}}(a(t_{*}))\|_p\le e^{-\alpha T(1-\frac{\mu}{\mu_0})}e^{\Lambda  t_*}\int_{t_n}^{t_*}\|I(s)\|_p\,ds
    \end{equation}
    by using \eqref{eq:spectral norm of D_U general} with $\gamma(s)=T-s$ and $\beta(s)=a(s)$. It follows that $\|U_{T-t_n}(a(t_n))-U_{T-t_{*}}(a(t_{*}))\|_p\to 0$ as $n\to\infty$, which is inconsistent. Letting now
    \begin{equation}
        c(t):= U_T(a_0)+\int_{0}^{t}DU_{T-s}(a(s))I(s)ds
    \end{equation}
    and deriving \eqref{eq:backward representation} almost everywhere w.r.t. $t$ yields
    \begin{equation}\label{eq:derivation a}
        \dot{a}(t)=-\partial_tV_{T-t}(c(t))+DV_{T-t}(c(t))DU_{T-t}(a(t))I(t)=N(V_{T-t}(c(t)))+I(t)=N(a(t))+I(t)
    \end{equation}
   by $a(t) = V_{T-t}(c(t))$ and $DV_{T-t}(c(t))DU_{T-t}(a(t))=\idty$. This completes the proof of the theorem. 
\end{proof}

\subsection{Proof of Proposition~\ref{pro:expansion for forward nominal-state synthesis}}\label{ss:proof of expansion for forward nominal-state synthesis}

\begin{proof}
   Theorem~\ref{thm:backward representation} ensures that the solution $a_\tau(\cdot)$ of \eqref{eq::nonlinear control flow notation} with the input $I_{sf, \tau}$ defined in \eqref{eq:initial step function} satisfies
    \begin{equation}
        a_\tau(T) = U_T(a_0)+\int_{T-\tau}^{T}DU_{T-t}(a_\tau(t)) I\, dt.
    \end{equation}
The function $\phi:t\mapsto\phi(t)=DU_{T-t}(a_\tau(t))I$ belongs to $L^1((0, T); L^p(\R^d))$ by Lemma~\ref{lem:general estimates flows}, and is derivable for almost every $t\in(0, T)$, and
\begin{equation}
    \dot{\phi}(t) = -DN(U_{T-t}(a_\tau(t)))DU_{T-t}(a_\tau(t))I+\partial DU_{T-t}(a_\tau(t))\dot{a}_\tau(t)I
\end{equation}
belongs to $L^1((0, T); L^p(\R^d))$ by \eqref{eq:operator norm of DN} and inequalities \eqref{eq:spectral norm of D_U general}, \eqref{eq:norm of a_dot(s)} and \eqref{eq:spectral norm of hessian D_U^2}. It follows that $\phi$ is absolutely continuous on $[0, T]$. By evoking \eqref{eq:Gâteaux derivative of DN} and \eqref{eq:derivative of Q_s^n}, a similar observation applies to $t\mapsto Z_{t,\tau,T}^nDU_{T-t}(a_\tau(t))I$ for any $n\in\N^{*}$, where $Z_{t,\tau,T}:=DN(U_{T-t}(a_\tau(t)))$. Via a double integration by parts and the fact that $a_\tau(T-\tau)=U_{T-\tau}(a_0)$, one gets
   \begin{eqnarray}\label{eq:integration by parts}
    a_\tau(T)
    &=& U_T(a_0)+\tau DU_\tau(U_{T-\tau}(a_0))I-\frac{\tau^2}{2}DN(U_T(a_0))DU_\tau(U_{T-\tau}(a_0))I+\int_{T-\tau}^{T}(t-T)\partial DU_{T-t}(a_\tau(t))\dot{a}_\tau(t)I\,dt\nonumber\\
    &&-\int_{T-\tau}^T\frac{(t-T)^2}{2}\left[\frac{d}{dt}Z_{t,\tau,T}\right]DU_{T-t}(a_\tau(t))I\,dt+\int_{T-\tau}^T\frac{(t-T)^2}{2}Z_{t,\tau,T}\partial DU_{T-t}(a_\tau(s))\dot{a}_\tau(t)I\,dt\nonumber\\
    &&+\int_{T-\tau}^T\frac{(t-T)^2}{2}Z_{t,\tau,T}^2DU_{T-t}(a_\tau(t))I\,dt
\end{eqnarray}
from which one performs another integration by parts on the last integral in \eqref{eq:integration by parts} and so on to obtain \eqref{eq:expansion for forward nominal-state synthesis} and \eqref{eq:kappa and chi}. Here $\partial DU_{T-t}(a_\tau(t))\dot{a}_\tau(t)I$ is the Gâteaux derivative of $DU_{T-t}$ at $a_\tau(t)$ in the direction $\dot{a}_\tau(t)$, and applied to $I$. Finally, the series \eqref{eq:expansion for forward nominal-state synthesis} and \eqref{eq:kappa and chi} are well-defined by Lemma~\ref{lem:well-defined kappa and chi}.
\end{proof}

Finally, the following result shows that the series in Proposition~\ref{pro:expansion for forward nominal-state synthesis} are well-defined. 

\begin{lemma}\label{lem:well-defined kappa and chi} 
     The series in \eqref{eq:expansion for forward nominal-state synthesis} is well-defined in $\mathscr{L}(L^p(\R^d))$ for every $p\in(1, \infty]$, and the series in \eqref{eq:kappa and chi} are well-defined in $\mathscr{L}(L^p(\R^d))$ for every $p\in[2, \infty]$. Moreover, it holds 
     \begin{equation}\label{eq:norm varphi}
         \|\varphi_{\tau,T}(I)\|_{\mathscr{L}(L^p(\R^d))}\le C \tau^3
     \end{equation}
     for some $ C:=C(\tau, T, \alpha,\mu, \mu_0, \|\omega\|_q,\|f''\|_\infty, \|a_0\|_p, \|I\|_p)>0$. 
\end{lemma}

\begin{proof}
First of all, the series in \eqref{eq:expansion for forward nominal-state synthesis} is well-defined in $\mathscr{L}(L^p(\R^d))$ by invoking the exponential operator~\eqref{eq:exponential operator}. 
    On the other hand, one has for every $n\ge 1$
    \begin{equation}\label{eq:derivative of Q_s^n}
        \frac{d}{dt}Z_{t,\tau,T}^n=\sum_{k=1}^{n}Z_{t,\tau,T}^{k-1}\dot{Z}_{t,\tau,T}Z_{t,\tau,T}^{n-k}
    \end{equation}
    where $\dot{Z}_{t,\tau,T}=-\partial DN(U_t(a_\tau(T-t)))DU_t(a_\tau(T-t))I$ using \eqref{eq:backward representation}. One deduces from \eqref{eq:derivative of Q_s^n} and \eqref{eq:Gâteaux derivative of DN}, and by applying \eqref{eq:spectral norm of D_U general} with $\gamma(t)=t$ and $\beta(t)=a_\tau(T-t)$ that for every $b\in L^p(\R^d)$,
\begin{equation}\label{eq:norm kappa}
    \left\|\kappa_{\tau,T}(I)\right\|_{\mathscr{L}(L^p(\R^d))}\le \Lambda_1\|I\|_p\max\left(1, e^{-2\alpha\tau\left(1-\frac{\mu}{\mu_0}\right)}\right)\sum_{n=1}^{\infty}\int_{0}^{\tau}\frac{t^{n+1}n\Lambda^n}{(n+1)!}\,dt\le C_0\tau^3
\end{equation}
for some $C_0:=C_0(\tau, \alpha,\mu, \mu_0, \|a_0\|_p, \|I\|_p,\|\omega\|_q,\|f''\|_\infty)>0$. Now, if $I\in L^p(\R^d)$ is constant in time, $\dot{a}_\tau(T-t)=-N(a_\tau(T-t))-I$ satisfies
\begin{equation}\label{eq:norm of a_dot(s)}
    \|\dot{a}_\tau(T-t)\|_p\le \Lambda  e^{-\alpha(T-t)(1-\frac{\mu}{\mu_0})}\|a_0\|_p+\Lambda (T-t)\max(1,e^{-\alpha (T-t)(1-\frac{\mu}{\mu_0})})\|I\|_p+\|I\|_p
\end{equation}
by \eqref{eq:a when p=finite}.
On the other hand, by applying \eqref{eq:spectral norm of D^2_phi} with $\gamma(t)=t$, $\beta(t)=a_\tau(T-t)$ and $h=\dot{a}_\tau(T-t)$, one gets,
\begin{equation}\label{eq:spectral norm of hessian D_U^2}
    \|\partial DU_t(a_\tau(T-t))\dot{a}_\tau(T-t)\|_{\mathscr{L}(L^p(\R^d)}\le \Lambda _1t e^{-\alpha t\left(1-\frac{\mu}{\mu_0}\right)}\|\dot{a}_\tau(T-t)\|_p\le C_1t
\end{equation}
for some $C_1:=C_1(\tau, T, \alpha,\mu, \mu_0, \|a_0\|_p, \|I\|_p,\|\omega\|_q,\|f''\|_\infty)$ using \eqref{eq:norm of a_dot(s)} and $0\le t\le\tau\le T$. It follows from $\|Z_{t,\tau,T}\|_{\mathscr{L}(L^p(\R^d))}\le\Lambda $ and \eqref{eq:spectral norm of hessian D_U^2} that
\begin{equation}\label{eq:norm chi}
     \left\|\chi_{\tau,T}(I)\right\|_{\mathscr{L}(L^p(\R^d))}\le C_1\sum_{n=1}^\infty\int_0^\tau\frac{t^{n+1}\Lambda^{n-1}}{(n+1)!}\,dt\le C_1e^{\Lambda\tau}\tau^3.
\end{equation}
    Combining \eqref{eq:norm kappa} and \eqref{eq:norm chi} yields \eqref{eq:norm varphi}.
\end{proof}

\subsection{Proof of~Theorem~\ref{thm:forward nominal-state synthesis}}\label{ss:proof of forward nominal-state synthesis}

\begin{proof}
Let us show that $a_\tau(T)=a_1$ is equivalent to \eqref{eq:implicit forward nominal-state synthesis}. Using \eqref{eq:expansion for forward nominal-state synthesis} and left multiplying $a_\tau(T)=a_1$ by $DN(U_T(a_0))$ yields
\begin{equation}\label{eq:derivation of the nonlinear input}
    DN(U_T(a_0))\left(a_1-U_T(a_0)\right)=-\cD_T(a_0)\varphi_{\tau,T}(I)I+\left(\idty-e^{-\tau\cD_T(a_0)}\right)DU_\tau(U_{T-\tau}(a_0))I,
\end{equation}
where $\cD_T(a_0):=DN(U_T(a_0))$. Since $[DU_\tau(U_{T-\tau}(a_0))]^{-1}=DV_\tau(U_T(a_0))$, one finds from \eqref{eq:derivation of the nonlinear input} that
\begin{equation*}
     DV_\tau(U_T(a_0))\left(\idty-e^{-\tau\cD_T(a_0)}\right)^{-1}\cD_T(a_0)\left(a_1-U_T(a_0)\right)=I-DV_\tau(U_T(a_0))\left(\idty-e^{-\tau\cD_T(a_0)}\right)^{-1}\cD_T(a_0)\varphi_{\tau,T}(I)I
\end{equation*}
which implies that $I\in L^p(\R^d)$ solves
\begin{equation}\label{eq:solving}
    \left[\idty-\cA_{\tau,T}(a_0)\varphi_{\tau,T}(I)\right]I=\cA_{\tau,T}(a_0)\left(a_1-U_T(a_0)\right)
\end{equation}
where $\cA_{\tau,T}(a_0)$ is defined by \eqref{eq:operator B_T}. This completes the proof of the sufficient part. Assume now that $I\in L^p(\R^d)$ satisfies \eqref{eq:implicit forward nominal-state synthesis}. It follows that
\begin{equation}\label{eq:solving 1}
   DN(U_T(a_0))\left(a_1-U_T(a_0)\right)=\left(e^{-\tau DN(U_T(a_0))}-\idty\right)DU_\tau(U_{T-\tau}(a_0))I-DN(U_T(a_0))\varphi_{\tau,T}(I)I.
\end{equation}
From \eqref{eq:expansion for forward nominal-state synthesis}, one finds
\begin{eqnarray}\label{eq:solving 2}
    DN(U_T(a_0))(a_\tau(T)-a_0)
    &=&\left(e^{-\tau DN(U_T(a_0))}-\idty\right)DU_\tau(U_{T-\tau}(a_0))I-DN(U_T(a_0))\varphi_{\tau,T}(I)I.
\end{eqnarray}
Identifying \eqref{eq:solving 1} with \eqref{eq:solving 2}, one deduces that $DN(U_T(a_0))(a_\tau(T)-a_1)=0$
so that $a_\tau(T)=a_1$ since $DN(U_T(a_0))$ is invertible. This completes the proof of the theorem.
\end{proof}

\subsection{Proof of Proposition~\ref{pro:expansion of varphi_tau,T}}\label{ss:proof of expansion of varphi_tau,T}

\begin{proof}
Let $I\in L^p(\R^d)$ and $I_{sf, \tau}(\cdot)$ be the step function defined by~\eqref{eq:initial step function}. Then, the corresponding solution $a_\tau(\cdot)$ to~\eqref{eq::nonlinear control flow notation} is given by
\[
a_\tau(t) = U_t(a_0),\qquad 0\le t\le T-\tau,\quad T\ge\tau>0,
\]
\[
a_\tau(t) = U_{t-T}\left(U_T(a_0)+b_\tau(t)\right),\qquad b_\tau(t):=\int_{T-\tau}^tDU_{T-s}(a(s))\,dsI,\qquad T-\tau<t\le T, \quad T\ge\tau>0
\]
by Theorem~\ref{thm:backward representation}. Applying~\eqref{eq:spectral norm of D_U general} with $\gamma(t)=T-t$ and $\beta(t)=a(t)$, one finds
\[
\|b_\tau(t)\|_p\le \|I\|_pe^{-\alpha T\left(1-\frac{\mu}{\mu_0}\right)}\int_{T-\tau}^te^{\alpha s\left(1+\frac{\mu}{\mu_0}\right)}\,ds\le \|I\|_pe^{-\alpha T\left(1-\frac{\mu}{\mu_0}\right)}e^{\alpha t\left(1+\frac{\mu}{\mu_0}\right)}\tau
\]
so that $b_\tau(t) = \cO(\tau)$ as $\tau\to 0$. It follows that
\[
a_\tau(t) = U_t(a_0)+\cO(\tau)\quad\text{as}\quad \tau\to 0,\qquad\forall t\in [0, T].
\]
Therefore, one deduces the following expansions, which hold for every $t\in [0, T]$:
\begin{equation*}
    \begin{split}
           Z_{t,\tau,T}&:=DN(U_t(a_\tau(T-t))) \underset{\tau \sim 0}{=}\Phi_T(a_0) +\cO(\tau),\qquad Q_{t,\tau,T}:=DU_t(a_\tau(T-t))\underset{\tau \sim 0}{=}DU_t(U_{T-t}(a_0))+\cO(\tau),\\
   P_{t,\tau,T}&:=\partial DU_t(a_\tau(T-t))\dot{a}_\tau(T-t) \underset{\tau \sim 0}{=} \partial DU_t(U_{T-t}(a_0))N(U_{T-t}(a_0))+\cO(\tau),
    \end{split}
\end{equation*}
where $\Phi_T(a_0):=DN(U_T(a_0))$. Thus, from~\eqref{eq:kappa and chi}, one gets
\begin{equation}\label{eq:expansion kappa and chi}
    \kappa_{\tau,T}(I)\underset{\tau \sim 0}{=}\cO(\tau^4),\qquad\chi_{\tau,T}(I)\underset{\tau \sim 0}{=}\chi_{\tau,T}(a_0) +\cO(\tau^4),
\end{equation}
where, using~\eqref{eq:link between the SAs abstract}, one has
\[
  \chi_{\tau,T}(a_0):=\sum_{n=1}^{\infty}\int_{0}^{\tau}\frac{(-t)^n\Phi_T(a_0)^{n-1}}{n!}\left[\Phi_T(a_0)DU_t(U_{T-t}(a_0))-DU_t(U_{T-t}(a_0))DN(U_{T-t}(a_0))\right]\,dt.
\]
Via integration by parts, one finds
\begin{eqnarray}\label{eq:chi tau, T, a0}
    \chi_{\tau,T}(a_0)&=&\int_0^\tau\left[e^{-t\Phi_T(a_0)}-\idty\right]DU_t(U_{T-t}(a_0))\,dt-\sum_{n=1}^\infty\int_0^\tau\frac{(-t)^n\Phi_T(a_0)}{n!}\left[\frac{d}{dt}DU_t(U_{T-t}(a_0))\right]\,dt\nonumber\\
    &=&\left[\idty-e^{-\tau\Phi_T(a_0)}\right]\Phi_T(a_0)^{-1}DU_\tau(U_{T-\tau}(a_0)+\left[e^{\tau\Phi_T(a_0)}-\idty\right]\Phi_T(a_0)^{-1}\nonumber\\
    &&-\underbrace{\sum_{n=1}^\infty\int_0^\tau\frac{(t-\tau)^{n+1}DU_t(U_{T-t}(a_0))}{(n+1)!}\left[\frac{d}{dt}Z_{t,T}^n\right]\,dt}_{\Lambda_{\tau,T}(a_0)},
\end{eqnarray}
where $Z_{t,T}:=DN(U_{T-t}(a_0))$. 
    Therefore, using an identity such as \eqref{eq:derivative of Q_s^n}, one deduces that
    \[
    \Lambda_{\tau,T}(a_0)\le C\tau^3,\quad\text{for some}\quad C:=C(\tau,T,\alpha,\mu,\|f''\|_\infty,\|\omega\|_q,\|a_0\|_p)>0
    \]
    by \eqref{eq:Gâteaux derivative of DN}, \eqref{eq:operator norm of DN} and \eqref{eq:spectral norm of D_U general}. It follows from~\eqref{eq:expansion kappa and chi} and~\eqref{eq:chi tau, T, a0} that $\varphi_{\tau,T}(I):=\kappa_{\tau,T}(I)+\chi_{\tau,T}(I)$ expands as in~\eqref{eq:expansion of varphi_tau,T}. Finally, the expansion~\eqref{eq:A_tau,T and varphi_tau,T} follows from~\eqref{eq:expansion of varphi_tau,T}, \eqref{eq:operator B_T}, and the fact that $\left[e^{\tau \Phi_T(a_0)}-\idty\right]^{-1}\underset{\tau \sim 0}{=}\cO(\tau^{-1})$ since $\Phi_T(a_0)$ is uniformly bounded with respect to $T$ and $a_0$.
\end{proof}

\section{Some useful properties of the drift and its associated flow operators}\label{s:complement results}

In the following, we prove some valuable properties of the nonlinear operator $N$ defined in \eqref{eq::operator N}.

\begin{proposition}\label{pro:second differential of N}
   The following holds. 
   \begin{enumerate}
       \item If $2\le p\le\infty$, then the Fréchet derivative $DN$ is Gâteaux differentiable at any $u\in L^p(\R^d)$, and for every $h\in L^p(\R^d)$, $\partial DN(u)h\in\mathscr{L}(L^p(\R^d))$ satisfies
        \begin{equation}\label{eq:Gâteaux derivative of DN}
      \begin{split}
 	\partial DN(u)hv&= \mu\omega\ast[f''(u)hv],\qquad\forall v\in L^p(\R^d)\\
\|\partial DN(u)h\|_{\mathscr{L}(L^p(\R^d))}&\le \mu\|f''\|_\infty\|\omega\|_q\|h\|_p.
  \end{split}
 \end{equation}
    \item Assume that $f$ is third times derivable and the third derivative $f^{(3)}$ is bounded. If $2< p\le\infty$, then $N\in C^2(L^p(\R^d); L^p(\R^d))$ and its second differential at $u\in L^p(\R^d)$ satisfies
        \begin{equation}\label{eq:second differential of N}
      \begin{split}
 	D^2N(u)hv=\partial DN(u)hv&= \mu\omega\ast[f''(u)hv],\qquad\forall (h,\, v)\in L^p(\R^d)\times L^p(\R^d)\\
\|D^2N(u)\|_{\mathscr{L}(L^p(\R^d)^2; L^p(\R^d))}&\le \mu\|f''\|_\infty\|\omega\|_q.
  \end{split}
 \end{equation}
 Here $q$ is the conjugate to $p$.
   \end{enumerate}
    
\end{proposition}
\begin{proof}
It is a consequence of \cite[Lemma~B.~8]{tamekue2024mathematical} that under Assumption~\ref{ass:general assumption}, $N\in C^1(L^p(\R^d); L^p(\R^d))$ for every $1<p\le\infty$.  It is straightforward to show that the Gâteaux derivative of $M:=DN$ at $u\in L^p(\R^d)$ in the direction $h\in L^p(\R^d)$ applied to $v\in L^p(\R^d)$ is given by 
\begin{equation}\label{eq:gateau differential of M}
    (\partial M(u)hv)(x) = \mu(\omega\ast[f''(u)hv])(x)=\mu\int_{\R^d}\omega(x-y)f''(u(y))h(y)v(y)dy,\qquad\forall x\in\R^d.
\end{equation}
One has that $\partial M(u)hv$ is bilinear with respect to $h$ and $v$. When $p=\infty$, the inequality in \eqref{eq:Gâteaux derivative of DN} is an immediate consequence of the Young-convolution inequality applied to \eqref{eq:gateau differential of M}. In the case of $2\le p<\infty$, one has
\begin{equation}\label{eq:norm of DM(u)hv}
    \|\partial M(u)hv\|_p^p=\mu^p\int_{\R^d}|\Gamma(x)|^pdx,\qquad\Gamma(x) := \int_{\R^d}\omega(x-y)f''(u(y))h(y)v(y)dy.
\end{equation}
  By Hölder inequality, and the generalized Jensen inequality, we find\footnote{Remembering that $\omega\in\cS(\R^d)$ or that $\omega\in L^1(\R^d)\cap L^p(\R^d)\cap L^q(\R^d)$.}
    \begin{equation}\label{eq:result of Jensen inequality to Gamma}
        |\Gamma(x)|^p\le\|f''\|_\infty^p\|\omega\|_q^p\left(\int_{\R^d}|h(y)|^p\frac{|\omega(x-y)|^q}{\|\omega\|_q^q}dy\right)\|v\|_p^p
    \end{equation}
 whenever $p/q\ge 1$, i.e. $2\le p<\infty$. Integrating \eqref{eq:result of Jensen inequality to Gamma} over $\R^d$, using Fubini theorem, one gets
 \begin{equation}\label{eq:norm of Gamma}
     \Gamma^p\le\|f''\|_\infty^p\|\omega\|_q^p\|h\|_p^p\|v\|_p^p.
 \end{equation}
 Therefore, \eqref{eq:Gâteaux derivative of DN} follows immediately by \eqref{eq:gateau differential of M}, \eqref{eq:norm of DM(u)hv} and \eqref{eq:norm of Gamma}.

 To prove that $M$ is differentiable (that is, $N$ is twice differentiable)--in the Fréchet sense--when $f$ is thrice derivable and $f^{(3)}$ is bounded, it suffices to prove that the Gâteaux derivative 
 \begin{equation}\label{eq:DM}
     \partial M:L^p(\R^d)\longrightarrow\mathscr{L}(L^p(\R^d)^2; L^p(\R^d)),\quad u\longmapsto \partial M(u)
 \end{equation}
is continuous. 
To this end, let $(u_n)\subset L^p(\R^d)$ be a real sequence converging in the $L^p$-norm to $u\in L^p(\R^d)$. We want to prove that $\partial M(u_n)$ converges to $\partial M(u)$ in $\mathscr{L}(L^p(\R^d)^2; L^p(\R^d))$. Let $(h,\, v)\in L^p(\R^d)^2$ and set
	\begin{equation}
		\Gamma_n:x\in\R^d\longmapsto \Gamma_n(x) = \int_{\R^d}\omega(x-y)[f''(u_n(y))-f''(u(y))]h(y)v(y)dy.
	\end{equation}
Since $f''$ is $\|f^{(3)}\|_\infty$-Lipschitz continuous, one immediately gets that 
 \[
 \|\Gamma_n\|_{\infty}\le\|f^{(3)}\|_\infty\|\omega\|_1\|u_n-u\|_\infty\|h\|_\infty\|v\|_\infty
 \]
by Young's convolution inequality. It follows that
\begin{equation}\label{eq::validity estimate infty}
     \|\partial M(u_n)-\partial M(u)\|_{\mathscr{L}(L^\infty(\R^d)^2; L^\infty(\R^d))} \le \mu\|f^{(3)}\|_\infty\|\omega\|_1\|u_n-u\|_\infty\xrightarrow[n\to\infty]{} 0.
\end{equation}
  Let us turn to an argument for $2< p<\infty$. Since $(u_n)$ tends in the $L^p$-norm to $u\in L^p(\R^d)$, for every $\varepsilon>0$, there exists $N\in\N$ such that for any $n\ge N$ it holds $\|u_n-u\|_p\le\varepsilon$.  We fix the previously defined $\varepsilon>0$ and $N$ in the following, and we consider $E_n:=\{y\in\R^d\mid |u_n(y)-u(y)|>\sqrt{\varepsilon}\}$. One has for every $x\in\R^d$,
 \begin{equation}\label{eq::separate}
      \Gamma_n(x) = \underbrace{\int_{\R^d\bs E_n}\hspace{-0.7cm}\omega(x-y)[f''(u_n(y))-f''(u(y))]h(y)v(y)dy}\limits_{\Lambda _1(x)}+\underbrace{\int_{E_n}\hspace{-0.3cm}\omega(x-y)[f''(u_n(y))-f''(u(y))]h(y)v(y)dy}\limits_{\Lambda _2(x)}.
 \end{equation}
By Chebyshev's inequality, it holds that
\begin{equation}\label{eq:cheb}
|E_n|\leq \frac{\|v_n-v\|_p^p}{\varepsilon^{\frac{p}2}}\leq \varepsilon^{\frac{p}2}
\end{equation}
 where $|E_n|$ denotes the Lebesgue measure of the measurable set $E_n\subset \R^d$.

\noindent On one hand, since $f''$ is $\|f^{(3)}\|_\infty$-Lipschitz continuous, one has by arguing as in the proof of \eqref{eq:norm of Gamma},
\begin{equation}\label{eq::separate 1}
  \|\Lambda _1\|_p\le \|f^{(3)}\|_\infty\|\omega\|_q\sqrt{\varepsilon}\|h\|_p\|v\|_p.
\end{equation}
On the other hand, using Hölder's inequality and the fact that $f''$ is bounded, one finds that
\begin{equation}
    |\Lambda _2(x)|^p\le 2^p\|f''\|_\infty^p|E_n|^{\frac pq-1}\int_{E_n}|\omega(x-y)|^p|h(y)|^pdy\|v\|_p^p,\qquad \forall x\in\R^d
\end{equation}
by a generalized Jensen's inequality. Integrating the above inequality with variable $x$ over $\R^d$, we find that
\begin{equation}\label{eq::separate 2}
 \|\Lambda _2\|_p:=\left\{\int_{\R^d} |\Lambda _2(x)|^p dx\right\}^{\frac{1}{p}}\le 2\|f''\|_\infty|E_n|^{\frac 1q-\frac 1p}\|\omega\|_p\|h\|_p\|v\|_p\le 2\|f''\|_\infty\|\omega\|_p\|h\|_p\|v\|_p(\sqrt{\varepsilon})^{p-2}
\end{equation}
by Fubini's theorem, where the last inequality is obtained owing to \eqref{eq:cheb} and $p/q = p-1$. Taking now the $L^p(\R^d)$-norm of both sides of \eqref{eq::separate}, applying Minkowski's inequality and using \eqref{eq::separate 1} and \eqref{eq::separate 2}, one gets that
\begin{equation}\label{eq::validity estimate}
     \|\partial M(u_n)-\partial M(u)\|_{\mathscr{L}(L^p(\R^d)^2; L^p(\R^d))} \le \mu(\|f^{(3)}\|_\infty\|\omega\|_q+2\|f''\|_\infty\|\omega\|_p)\max(\sqrt{\varepsilon}, (\sqrt{\varepsilon})^{p-2}). 
\end{equation}
It follows that the operator $\partial M$ defined in \eqref{eq:DM} is continuous, and \eqref{eq:second differential of N} follows immediately.
\end{proof}

The following lemma is more general than \cite[Lemma~B.~10]{tamekue2024mathematical} given
that the operator norm of $DU_t(u)$ does not systematically grow with $t\ge 0$ as one might think by examining \cite[Lemma~B.~10]{tamekue2024mathematical}.

\begin{lemma}\label{lem:general estimates flows}
        Let $p\in(1, \infty]$, $\beta\in C^0(\R_+;L^p(\R^d))$, and $\gamma\in C^1(\R_+; \R)$ with $\gamma(0)\ge 0$. Then, for every $t\ge 0$, it holds that
\begin{equation}\label{eq:spectral norm of D_U general}
    \|DU_{\gamma(t)}(\beta(t))\|_{\mathscr{L}(L^p(\R^d))}\le\displaystyle e^{-\alpha\gamma(0)\left(1-\frac{\mu}{\mu_0}\right)}e^{-\alpha\int_0^t\dot{\gamma}(\tau)\,d\tau}e^{\alpha\frac{\mu}{\mu_0}\int_0^t|\dot{\gamma}(\tau)|\,d\tau}.
\end{equation}

If $p\in[2, \infty]$, then for all $h\in L^p(\R^d)$ and every $t\ge 0$, it holds that
\begin{equation}\label{eq:spectral norm of D^2_phi}
\|\partial DU_{\gamma(t)}(\beta(t))h\|_{\mathscr{L}(L^p(\R^d))}\le \Lambda _1 e^{-\alpha\int_0^t\dot{\gamma}(\tau)\,d\tau}e^{\alpha\frac{\mu}{\mu_0}\int_0^t|\dot{\gamma}(\tau)|\,d\tau}e^{-\alpha\left(1-\frac{\mu}{\mu_0}\right)\gamma(0)}\left(\|\dot{\gamma}\|_\infty t+\gamma(0)\right)\|h\|_p
\end{equation}
 where $\|\dot{\gamma}\|_\infty:=\max_{0\le s\le t}|\dot{\gamma}(s)|$, $\Lambda _1:=\mu\|f''\|_\infty\|\omega\|_q$ and $\partial DU_{\gamma(t)}$ is the Gâteaux derivative of $DU_{\gamma(t)}$.
 
In the specific case of $p=2$, one can sharpen \eqref{eq:spectral norm of D_U general} and \eqref{eq:spectral norm of D^2_phi} by replacing $\mu_0$ by $\mu_1$.
\end{lemma}
\begin{proof}
    
    Let $v,\,h\in L^p(\R^d)$ and $\gamma\in C^1(\R_+,\R)$ such that $\gamma(0)\ge 0$. It follows from \eqref{eq:equation derivative of the flow} that
    \begin{equation}\label{eq:equation derivative of U_alpha}
    \dot{y}(t) = -\alpha\dot{\gamma}(t)y(t)+\mu\dot{\gamma}(t)\omega\ast[f'(U_{\gamma(t)}(v))y(t)],\quad y(t):=DU_{\gamma(t)}(v)h,\quad\forall t\ge 0
\end{equation}
which yields the following integral representation
\begin{equation}
    y(t) = e^{-\alpha\int_0^t\dot{\gamma}(\tau)\,d\tau}y(0)+\mu\int_0^t\dot{\gamma}(s)e^{-\alpha\int_0^t\dot{\gamma}(\tau)\,d\tau}e^{\alpha\int_0^s\dot{\gamma}(\tau)\,d\tau}\omega\ast[f'(U_{\gamma(s)}(v))y(s)]\,ds.
\end{equation}
Taking the $L^\infty(\R^d)$-norm and Gronwall's lemma yields
\begin{equation}\label{eq:with gamma}
    \|y(t)\|_\infty\le e^{-\alpha\int_0^t\dot{\gamma}(\tau)\,d\tau}e^{\alpha\frac{\mu}{\mu_0}\int_0^t|\dot{\gamma}(\tau)|\,d\tau}e^{-\alpha\left(1-\frac{\mu}{\mu_0}\right)\gamma(0)}\|h\|_\infty.
\end{equation}
Let $\beta\in C^0(\R_+;L^\infty(\R^d))$, then for a fixed $t\ge 0$, the map $s\mapsto DU_{\gamma(t)}(\beta(s))$ is continuous by composition, and using \eqref{eq:with gamma}, we find that
\begin{equation}
    \|DU_{\gamma(t)}(\beta(s))\|_{\mathscr{L}(L^\infty(\R^d))}\le e^{-\alpha\int_0^t\dot{\gamma}(\tau)\,d\tau}e^{\alpha\frac{\mu}{\mu_0}\int_0^t|\dot{\gamma}(\tau)|\,d\tau}e^{-\alpha\left(1-\frac{\mu}{\mu_0}\right)\gamma(0)},\qquad\forall s\in[0, t)
\end{equation}
which proves \eqref{eq:spectral norm of D_U general} by continuity and taking the $\lim\sup$ as $s\to t$. 

If $1<p<\infty$, one uses the last paragraph of Notation~\ref{not:general} as follows. One has from \eqref{eq:equation derivative of U_alpha},
\begin{equation}\label{eq:cases of 1<p<infty y(t)}
     \frac{1}{2}\frac{d}{dt}\|y(t)\|_p^2=D\varphi(y(t))\dot{y}(t) = \langle y(t)^{*},\dot{y}(t)\rangle_{L^q, L^p}\le (-\alpha\dot{\gamma}(t)+\mu\|\omega\|_1|\dot{\gamma}(t)|)\|y(t)\|_p^2
\end{equation}
    by Hölder and Young's convolution inequalities. It follows that
\begin{equation}
    \|y(t)\|_p\le e^{-\alpha\int_0^t\dot{\gamma}(\tau)\,d\tau}e^{\alpha\frac{\mu}{\mu_0}\int_0^t|\dot{\gamma}(\tau)|\,d\tau}\|y(0)\|_p,\qquad\forall t\ge 0
\end{equation}
by Gronwall's lemma. One concludes the proof by arguing as in the case of $p=\infty$.

 Beside the above analysis that contains the case of $p=2$, by taking advantage of the Fourier transform properties in $L^2(\R^d)$, one gets
 \begin{equation}
     \frac{1}{2}\frac{d}{dt}\|y(s)\|_2^2=\langle y(t), DN(U_{\gamma(t)}(v))y(t)\rangle_{L^2}\le -\alpha\dot{\gamma}(t)\|y(t)\|_2^2+\mu|\dot{\gamma}(t)|\|\widehat{\omega}\|_\infty\|y(t)\|_2^2
 \end{equation}
    by Parseval's identity and Plancherel's theorem. Inequalities~\eqref{eq:spectral norm of D_U general} follow by arguing as previously.

Now, letting $y(t)$ as previously, $z(t):=\partial DU_{\gamma(t)}(v)h$, and using \eqref{eq:equation derivative of the flow}, we find that
\begin{equation}
    \dot{z}(t) = -\alpha\dot{\gamma}(t)z(t)+\mu\dot{\gamma}(t)\omega\ast[f'(U_{\gamma(t)}(v))z(t)]+\dot{\gamma}(t)\partial DN(U_{\gamma(t)}(v))y(t),\quad z(0)=\partial DU_{\gamma(0)}(v)h.
\end{equation}
Taking the $L^p(\R^d)$-norm (immediate in the case of $p=\infty$, while for $2\le p<\infty$, one can argue as previously) and using \eqref{eq:Gâteaux derivative of DN} and \eqref{eq:with gamma}, we find
by Gronwall's lemma,
\begin{equation}
    \|\partial DU_{\gamma(t)}(v)h\|_{\mathscr{L}(L^p(\R^d))}\le \Lambda _1 e^{-\alpha\int_0^t\dot{\gamma}(\tau)\,d\tau}e^{\alpha\frac{\mu}{\mu_0}\int_0^t|\dot{\gamma}(\tau)|\,d\tau}e^{-\alpha\left(1-\frac{\mu}{\mu_0}\right)\gamma(0)}\left(\|\dot{\gamma}\|_\infty t+\gamma(0)\right)\|h\|_p.
\end{equation}
One completes the proof of \eqref{eq:spectral norm of D^2_phi} by arguing as previously. In the specific case of $p=2$, one proves \eqref{eq:spectral norm of D^2_phi} by arguing as in the proof of \eqref{eq:spectral norm of D_U general} above.
\end{proof}

We also have the following key results.

\begin{lemma}\label{lem:norm of exp A_t(U_t(a0))}
    Let $p\in(1, \infty]$, $a\in L^p(\R^d)$, and let $s,\,t\in\R_{+}$. Let $\Phi_t(a):=DN(U_t(a))$ where $DN$ is the Fréchet derivative of $N$, and $U_t$ its flow at $t$. Then, it holds that
    \begin{equation}\label{eq:norm of exp A_t(U_t(a0))}
        \|e^{t\Phi_s(a)}\|_{\mathscr{L}(L^p(\R^d))}\begin{cases}
            \le e^{-\alpha\left(1-\frac{\mu}{\mu_0}\right)t}&\quad\mbox{if}\quad 1< p\le\infty,\\
            \le e^{-\alpha\left(1-\frac{\mu}{\mu_1}\right)t}&\quad\mbox{if}\quad p=2.
        \end{cases}
    \end{equation}
\end{lemma}
\begin{proof}
    Fix $s\in\R_{+}$, let $u\in L^p(\R^d)$, and define $c(t)=e^{t\Phi_s(a)}u$, for $t\in\R_{+}$. Then, $c\in C^1(\R_{+}; L^p(\R^d))$, $c(0)=u$, and it holds that
    \begin{equation}\label{eq:derivative of c(t)}
        \dot{c}(t) = \Phi_s(a)c(t)= -\alpha c(t)+\mu\omega\ast[f'(U_s(a))c(t)].
    \end{equation}

    Assume that $p=\infty$, and represent \eqref{eq:derivative of c(t)} under its integral form by
    \begin{equation}\label{eq:integral representation of c}
        c(t)=e^{-\alpha t}u+\mu\int_{0}^{t}e^{-\alpha(t-\tau)}\omega\ast[f'(U_s(a))c(\tau)]d\tau.
    \end{equation}
    Since $\|f'(U_s(a))\|_\infty\le 1$, taking the $L^\infty(\R^d)$-norm of \eqref{eq:integral representation of c}, and using Gronwall's lemma, we find
    \begin{equation}
        \|e^{t\Phi_s(a)}u\|_\infty\le e^{-\alpha\left(1-\frac{\mu}{\mu_0}\right)t}\|u\|_\infty\quad\forall s,\, t\in\R_+.
    \end{equation}
    In the cases of $1< p<\infty$, we introduce for every $t\in\R_{+}$,  $g(t)=\|c(t)\|_p^2$, where $c$ is defined as previously. Therefore, $g\in C^1(\R_{+})$, $g(0)=\|u\|_p^2$, and owing to the last paragraph of Notation~\ref{not:general}, we get from \eqref{eq:derivative of c(t)} that 
    \begin{eqnarray}\label{eq:derivative of map g}
        \frac{1}{2}\dot{g}(t)&=&\left\langle c(t)^{*},\dot{c}(t)\right\rangle_{L^q, L^p}
        \le-\alpha\|c(t)\|_p^{2}+\mu\|c(t)\|_p^{2-p}\|c(t)\|_p^{p-1}\|\omega\|_1\|c(t)\|_p=-(\alpha-\mu\|\omega\|_1)g(t)
    \end{eqnarray}
 by Hölder and Young's convolution inequalities. Gronwall's inequality to \eqref{eq:derivative of map g} yields the desired result. Finally, for $p=2$, one uses a Fourier transform argument as in the proof of Lemma~\ref{lem:general estimates flows} to replace \eqref{eq:derivative of map g} by
    \begin{equation}\label{eq:Gamma 1 p=2}
     \dot{g}(t)\le -2(\alpha-\mu\|\widehat{\omega}\|_\infty)g(t).
    \end{equation} 
    This completes the proof of the lemma.
\end{proof}

The proof of the following can be done by mimicking its finite-dimension counterpart~\cite[Lemma~C.4]{tamekue2024control}.

\begin{lemma}\label{lem:on the SC}
Let $p \in [2, \infty]$, $a \in L^p(\mathbb{R}^d)$, and let $T > 0$, $t \in [0, T]$. The following identity holds,
\begin{equation}\label{eq:link between the SAs abstract}
DN(U_T(a))\, DU_t(U_{T-t}(a)) 
= DU_t(U_{T-t}(a)) \, DN(U_{T-t}(a))
+ \partial DU_t(U_{T-t}(a)) \, N(U_{T-t}(a)),
\end{equation}
where $\partial DU_{T - t}$ denotes the Gâteaux derivative of the Fréchet derivative $DU_t$ at $U_{T - t}(a)$.

\end{lemma}

\bibliographystyle{siamplain}
\bibliography{references}
\end{document}